\documentclass[10pt]{amsart}

\usepackage{amsmath}
\usepackage{amssymb}
\usepackage{amscd}
\usepackage{latexsym}
\usepackage[latin1]{inputenc}

\usepackage{graphicx}
\usepackage{changebar,color}

\newtheorem{theorem}{Theorem}[section]
\newtheorem{corollary}[theorem]{Corollary}
\newtheorem{lemma}[theorem]{Lemma}
\newtheorem{conjecture}[theorem]{Conjecture}
\newtheorem{remark}[theorem]{Remark}
\newtheorem{proposition}[theorem]{Proposition}

\newtheorem{definition}[theorem]{Definition}

\numberwithin{equation}{section}

\newcommand{\Hom}{\operatorname{Hom}}
\newcommand{\Ker}{\operatorname{Ker}}
\newcommand{\cl}{\operatorname{cl}}

\def\N{\mathbb{N}}
\def\Z{\mathbb{Z}}
\def\PP{\mathbb{P}}
\def\Q{\mathbb{Q}}

\def\C{\mathbb{C}}

\def\G{\mathbb{G}}
\def\H{\mathbb{H}}

\newcommand{\Qb}{{\overline {\mathbb Q}}}
\newcommand{\Gm}{{{\mathbb G}_m}}
\newcommand{\Ga}{{{\mathbb G}_a}}

\newcommand{\cf}{{\em cf. }}

\newcommand{\cO}{{\mathcal O}}

\newcommand{\CdRB}{{\mathcal C}_{\rm dRB}}
\newcommand{\CdRBQ}{{\mathcal C}_{\rm dRB, \Q}}

\newcommand{\Lie}{{\rm Lie\,}}

\newcommand{\pr}{{\rm pr}}
\newcommand{\Alb}{{\rm Alb}}
\newcommand{\alb}{{\rm alb}}
\newcommand{\Gr}{{\rm Gr}}

\newcommand{\Pic}{{\rm Pic}}
\newcommand{\Res}{{\rm Res}}

\newcommand{\rB}{{\rm B}}

\newcommand{\oli}{\overline}

\newcommand{\im}{{\rm im\,}}

\newcommand{\Biext}{{\rm Biext}}
\newcommand{\SymBiext}{{\rm SymBiext}}

\newcommand{\an}{{\rm an}}
\newcommand{\dR}{{\rm dR}}
\newcommand{\dRB}{{\rm dRB}}

\newcommand{\Tr}{{\rm Tr}}

\newcommand{\LiePer}{{\rm LiePer}}
\newcommand{\Per}{{\rm Per}\,}

\newcommand{\HdR}{H_{\rm dR}}

\newcommand{\ra}{\rightarrow}
\newcommand{\lrasim}{\stackrel{\sim}{\longrightarrow}}
\newcommand{\lra}{\longrightarrow}
\newcommand{\hra}{\hookrightarrow}

\newcommand{\hlra}{{\lhook\joinrel\longrightarrow}}

\newcommand{\ggp}{\rm{gp}}

\date{\today}

\title[Concerning the Grothendieck Period Conjecture]{Some remarks concerning \\ the Grothendieck Period Conjecture} 
\author{Jean-Beno\^{\i}t Bost}
\address{Jean-Beno\^{\i}t  Bost, D{é}partement de Math{é}matique, Universit{é}
Paris-Sud,
B{â}timent 425, 91405 Orsay cedex, France}
\email{jean-benoit.bost@math.u-psud.fr}
\author{Fran\c cois Charles}
\address{Fran\c cois Charles, Laboratoire de mathématiques d'Orsay, UMR 8628 du CNRS, Universit{é}
Paris-Sud,
B{â}timent 425, 91405 Orsay cedex, France}
\email{francois.charles@math.u-psud.fr}
\begin{document}
\begin{abstract}
We discuss various results and questions around the Grothendieck period conjecture, which is a counterpart, concerning the de Rham-Betti realization of algebraic varieties over number fields,  of the classical conjectures of Hodge and Tate. These results give new evidence towards the conjectures of Grothendieck and Kontsevich-Zagier concerning transcendence properties of the torsors of periods of varieties over number fields.

Let $\Qb$ be the algebraic closure of $\Q$ in $\C$, let $X$ be a smooth projective variety over $\Qb$ and let $X^\an_\C$ denote the compact complex analytic manifold that it defines. The Grothen\-dieck Period Conjecture in codimension $k$ on $X$, denoted $GPC^k(X)$, asserts that any class $\alpha$ in the algebraic de Rham cohomology group $H^{2k}_{\dR}(X/\Qb)$ of $X$ over $\Qb$ such that
$$\frac{1}{(2\pi i)^k} \int_\gamma \alpha \in \Q$$
for every rational homology class $\gamma$ in $H_{2k}(X^\an_\C, \Q)$ is the class in algebraic de Rham cohomology of some algebraic cycle of codimension $k$  in $X$, with rational coefficients.  

We notably establish that $GPC^1(X)$ holds when $X$ is a product of curves, of abelian varieties, and of $K3$ surfaces, and that $GPC^2(X)$ holds for a smooth cubic hypersurface $X$ in $\PP^5_\Qb$. We also discuss the conjectural relationship of Grothendieck classes with the weight filtration on cohomology.
\end{abstract}

%
%
\maketitle

\setcounter{tocdepth}{2}
{
\tableofcontents
}

In this article, $\Qb$ denotes the algebraic closure of $\Q$ in $\C$.

Let $X$ be a smooth projective variety over $\Qb$ and let $X^\an_\C$ denote the compact complex analytic manifold defined by the set of complex points of the smooth projective complex variety $X_\C.$
 If a cohomology class $\beta$ in $H^{2k}(X^\an_\C, \Q)$ is algebraic --- in other words, if $\beta$ is the class of some algebraic cycle of codimension $k$ in $X_\C$, or equivalently in $X$, with rational coefficients --- then the class $(2\pi i)^k \beta$ in $H^{2k}(X^\an_\C, \C)$ belongs to the $\Qb$-vector subspace $H^{2k}_{\dR}(X/\Qb)$ of $H^{2k}(X^\an_\C, \C)$ defined by the algebraic de Rham cohomology of $X$ over $\Qb$. 
 
 The \emph{Grothendieck Period Conjecture} $GPC^k(X)$  claims that, conversely, any cohomology class $\beta$  in  $H^{2k}(X^\an_\C, \Q)$  such that $(2\pi i)^k \beta$ belongs to $H^{2k}_{\dR}(X/\Qb)$ is algebraic.

 This work is mainly devoted to the codimension 1 case of this conjecture. We investigate this case by combining transcendence results  on commutative algebraic groups derived from the transcendence theorems of Schneider and Lang and diverse geometric constructions inspired by the ``philosophy of motives". Our transcendence arguments elaborate on the ones in \cite{Bost12}, and the motivic ones are variations on arguments classical in the study of absolute Hodge classes and of the conjectures of Hodge and Tate.
 
 Notably we establish the validity of $GPC^1(X)$ when $X$ is a product of curves, of abelian varieties, and of $K3$ surfaces (or more generally of smooth projective hyperk\"ahler varieties with second Betti number at least $4$)  over $\Qb$. This allows us to show that $GPC^2(X)$ holds for a smooth cubic hypersurface $X$ in $\PP^5_\Qb$.

\section{Introduction}

\subsection{The conjecture $GPC^k(X)$}\label{subGPC}

Let $X$ be a smooth projective variety\footnote{By a \emph{variety} over some field $k$, we mean a geometrically integral separated scheme of finite type over $k$.} over $\Qb.$

\subsubsection{De Rham and Betti cohomology groups}\label{dRBcomp} We refer the reader to \cite{Grothendieck66, Hartshorne75, CharlesSchnell11} for additional references and details on the basic facts recalled in this paragraph. 

To $X$ are attached its \emph{algebraic de Rham cohomology groups}, defined as the hypercohomology groups
$$\HdR^i(X/\Qb) := \H^i (X, \Omega^\bullet_{X/\Qb})$$ 
of the algebraic de Rham complex
$$\Omega^\bullet_{X/\Qb} : 0 \longrightarrow
\Omega^0_{X/\Qb} =\cO_{X} \stackrel{d}{\longrightarrow}
\Omega^1_{X/\Qb} \stackrel{d}{\longrightarrow}
\Omega^2_{X/\Qb} \stackrel{d}{\longrightarrow} \cdots  .
$$
We may also consider the compact connected complex analytic manifold $X^\an_\C$ defined by the smooth projective variety $X_\C$ over $\C$ deduced from $X$ by extending the base field from  $\Qb$ to $\C$,
 and its \emph{Betti cohomology groups}
 $H^i(X^\an_\C, \Q).$
 
 The base change $\Qb \hlra \C$ defines a canonical isomorphism
 \begin{equation}\label{dRQC}
\HdR^i(X/\Qb)\otimes_\Qb \C \lrasim \HdR^i(X_\C/\C):= \H^i (X_\C, \Omega^\bullet_{X_\C/\C}),
\end{equation}
and the GAGA comparison theorem shows that ``analytification" defines an isomorphism
\begin{equation}\label{GAGA}
 \H^i (X_\C, \Omega^\bullet_{X_\C/\C}) \lrasim \H^i (X^\an_\C, \Omega^\bullet_{X^\an_\C}),
\end{equation}
where $\Omega^\bullet_{X^\an_\C}$ denotes the analytic de Rham complex:
$$\Omega^\bullet_{X^\an_\C} : 0 \longrightarrow
\cO_{X^\an_\C} \stackrel{d}{\longrightarrow}
\Omega^1_{X^\an_\C} \stackrel{d}{\longrightarrow}
\Omega^2_{X^\an_\C} \stackrel{d}{\longrightarrow} \cdots  .
$$
Finally the analytic Poincar\'e lemma shows that the injective morphism of sheaves  $\C_{X^\an_{C}} \hlra \cO_{X^\an_{\C}}$ on $X^\an_\C$ defines a quasi-isomorphism of complexes of abelian sheaves
$\C_{X^\an_\C} \stackrel{q.i.}{\lra} \Omega^\bullet_{X^\an_\C},$
and consequently an isomorphism of (hyper)cohomology groups:
\begin{equation}\label{dRan}
H^i(X^\an_\C, \C) \lrasim \H^i (X^\an_\C, \Omega^\bullet_{X^\an_\C}).
\end{equation}
The composition of (\ref{dRQC}), (\ref{GAGA}), and of the inverse of (\ref{dRan}) defines a natural comparison isomorphism:
\begin{equation}\label{compdRC}
\HdR^i(X/\Qb)\otimes_\Qb \C \lrasim H^i (X^\an_\C,  \C).
\end{equation}

Besides, the extension of fields of coefficients $\Q \,\hlra\, \C$ defines a natural isomorphism:
\begin{equation}\label{compBC}
H^i (X^\an_\C, \Q)\otimes_\Q \C \lrasim H^i (X^\an_\C, \C).
\end{equation}

In this article, the isomorphisms (\ref{compdRC}) and (\ref{compBC}) will in general be written as equalities. For instance, for any element $\alpha$ in $\HdR^i(X/\Qb)$ (resp. $\beta$ in $H^i (X^\an_\C, \Q)$), its image by the inclusion $\HdR^i (X/\Qb) \hlra H^i (X^\an_\C, \C)$ (resp.  
$H^i (X^\an_\C, \Q)\hlra H^i (X^\an_\C, \C$)) determined by  (\ref{compdRC}) (resp. (\ref{compBC})) will be denoted $\alpha \otimes_\Qb 1_\C$ (resp. $\beta \otimes_\Q 1_\C$), or even $\alpha$ (resp. $\beta$) when no confusion may arise. 

\subsubsection{Cycle maps} Recall that there is a canonical way of associating a class $\cl^X_\dR(Z)$ in $H^{2k}_{\dR}(X/\Qb)$ with any element $Z$ of the group $Z^k(X)$ of algebraic cycles on $X$ of pure codimension $k$ (see for instance \cite{Hartshorne75}, II.7, and \cite{DeligneMilneOgusShih}, I.1). This construction defines cycle maps
$$\cl^X_{\dR}:  Z^k(X) \lra H^{2k}_{\dR}(X/\Qb).$$
These maps are compatible with algebraic equivalence and intersection products. They are functorial and compatible with Gysin maps.

When $k=1$, the cycle $Z$ is a divisor on $X$ and $\cl^X_\dR(Z)$ may  be defined as the image of the class of $\cO_X(Z)$ in $\Pic(X) \simeq
H^1(X, \cO_C^\times)$ by the map 
$$c_{1,\dR}:  H^1(X, \cO_X^\times) \lra H^2_\dR(X/\Qb)$$
induced in (hyper)cohomology by the morphism of (complex of) sheaves
$$
\begin{array}{rrclcc}
d\log: & \cO_X^\times  & \lra  & \Omega_{X/\Qb}^{1, d=0}& \hlra & \Omega^\bullet_{X/\Qb} [1]   \\
 & f & \longmapsto   & f^{-1}. df. &  & \end{array}
$$
Starting from $c_{1,\dR}$, one may define Chern classes $c_{k,\dR}$ of vector bundles, and  consequently of coherent $\cO_X$-modules, over $X$. Then the class of any closed integral subscheme $Z$ of codimension $k$ in $X$ is given by
$$\cl^X_\dR(Z) := \frac{(-1)^{k-1}}{(k-1)!} c_{k,\dR}(\cO_Z).$$

Similarly, using Chern classes in Betti cohomology, one defines ``topological" cycles maps
$$\cl^X_{\rB}:  Z^k(X_\C) \lra H^{2k}(X^{\an}_\C,\Q).$$
We refer the reader to \cite{Voisin2002}, Chapter 11, for a discussion of alternative constructions of the cycle class $\cl^X_{\rB}(Z)$ attached to a cycle  $Z$ in $Z^k(X_\C)$, notably in terms of the integration current $\delta_Z$ on $X^\an_\C$.

Occasionally, when no confusion may arise, we shall simply denote $[Z]$ the cycle class of a cycle $Z$ in de Rham or in Betti cohomology.

Up to a twist by some power of $2\pi i$, the above two constructions of cycle classes are compatible: 

\begin{proposition}\label{compdRB} For any integer $k$ and any cycle $Z$ in $Z^k(X)$, the following equality holds in $H^{2k}(X^\an_{\C}, \C)$:
\begin{equation}\label{competa}
\cl^X_{\dR}(Z)\otimes_{\Qb}1_{\C} = \epsilon_{k,d} (2\pi i)^k \cl^X_{\rB}(Z_{\C}) \otimes_{\Q}1_\C,
\end{equation}
where $\epsilon_{k,d}$ denotes  a sign\footnote{This sign is a function of $k$ and $d:= \dim X$ only, depending on the sign conventions used in the constructions of the cycle maps $\cl^X_{\dR}$ and $\cl^X_{\rB}$.}.
\end{proposition}

For $k=1$, that is, for the first Chern class, this is a straightforward consequence of the definitions (see for instance  \cite{Deligne71}, 2.2.5). This special case implies the general one by the general formalism of Chern classes.

\subsubsection{The conjecture $GPC^k(X)$}\label{statement}

As indicated at the end of \ref{dRBcomp}, we shall write the canonical injections
$$H^i (X^\an_\C, \Q) \hlra H^i (X^\an_\C, \Q)\otimes_\Q \C \lrasim H^i (X^\an_\C, \C)$$
and 
$$\HdR^i(X/\Qb) \hlra
\HdR^i(X/\Qb)\otimes_\Qb \C \lrasim H^i (X^\an_\C,  \C)$$
as inclusions. For any integer $k$, we also consider the space
$$H^i (X^\an_\C, \Q(k)):= H^i (X^\an_\C, (2\pi i)^k\Q),$$
and we identify it to the subspace $(2 \pi i)^k H^i (X^\an_\C, \Q)$ of $H^i (X^\an_\C, \C)$.

According to these conventions, the relation (\ref{competa}) may be written
$$\cl^X_{\dR}(Z) = \epsilon_{k,d} (2\pi i)^k \cl^X_{\rB}(Z_{\C}),$$
and shows that the image of $\cl^X_\dR$ lies in the finite-dimensional $\Q$-vector space
$$H^{2k}_{\Gr}(X, \Q(k)) := H^{2k}_{\dR}(X/\Qb) \cap H^{2k}(X^\an_{\C}, \Q(k)).$$
These groups depend functorially on $X$: to any morphism $f: X \lra Y$ of smooth projective varieties over $\Qb$ one can attach a $\Q$-linear pull-back map
$$f^\ast_{\Gr}: H^{2k}_{\Gr}(Y, \Q(k)) \lra H^{2k}_{\Gr}(X,\Q(k)),$$
defined by the pull-back maps $f^\ast_\dR$ and $(f^{\an}_{\C,\rB})^\ast$ in algebraic de Rham and Betti cohomology.
  
The cycle class map $\cl^X_{\dR} =\epsilon_{k,d} (2\pi i)^k \cl^X_{\rB}$ from $Z^k(X)$ to $H^{2k}_{\Gr}(X, \Q(k))$ extends uniquely to a $\Q$-linear map
$$\cl^X_{\Gr}: Z^k(X)_{\Q} \lra H^{2k}_{\Gr}(X, \Q(k)),$$
and the \emph{Grothendieck Period Conjecture for cycles of codimension $k$ in $X$} is the assertion:
\begin{center}
$GPC^k(X)$: \emph{the morphism of $\Q$-vector spaces $\cl^X_{\Gr}: Z^k(X)_{\Q} \lra H^{2k}_{\Gr}(X, \Q(k))$ is onto.} 
\end{center}

This assertion characterizes --- conjecturally --- the cohomology
classes  with rational coefficients of algebraic cycles in $X$ by their joint rationality properties in the de Rham cohomology of $X/\Qb$ and in the Betti cohomology of $X^\an_{\C}.$

Observe that since Hilbert schemes of subschemes of $X$ are defined over $\Qb$,  $Z^k(X_\C)$ and its subgroup $Z^k(X)$ have the same image in $H^{2k}(X^\an_\C, \Z)$ by the cycle class map $\cl^X_{\rB}$ and that, according to Proposition \ref{compdRB}, the surjectivity of the cycle map
$$\cl^X_{\dR, \Qb}:  Z^k(X)_\Qb \lra H^{2k}_{\dR}(X/\Qb)$$
and the one of 
$$\cl^X_{\rB, \Q}:  Z^k(X_\C)_\Q \lra H^{2k}(X^{\an}_\C,\Q)$$
are equivalent. 
Therefore, when these cycle maps are surjective, $GPC^k(X)$ is true and  
$$H^{2k}_{\Gr}(X, \Q(k)) = H^{2k}(X^\an_{\C}, \Q(k)).$$

This discussion applies trivially when $k =0$ or $k= \dim X$ --- in particular $GPC^1(X)$ holds for any smooth projective curve $X$ over $\Qb$ --- and for any $k$ when $X$ is a cellular variety, for instance a Grassmannian (\cf \cite{FultonIT}, Examples 1.9.1 and 19.1.11). 

Also observe that, as a straightforward consequence of the hard Lefschetz theorem, if $X$ is a smooth projective variety over $\Qb$ of dimension $n$ and if $2k\leq n$, the following implication holds:
$$GPC^k(X)\implies GPC^{n-k}(X).$$

The Grothendieck Period  Conjecture is mentioned briefly in \cite{Grothendieck66} (note (10), p. 102) and with more details in \cite{Lang66b} (Historical Note of Chapter IV). It is presented by Andr\'e in his monographs \cite{Andre89} (IX.2.2) and \cite{AndreMotives04} (Section 7.5).  See Section \ref{Period} \emph{infra} for a  discussion of the relation between the original formulation of Grothendieck period conjecture and the conjectures $GPC^k(X)$ considered in this article.

\subsection{Summary of our results}\label{subRes} In \cite{Bost12}, Section 5, the conjecture $GPC^1(X)$ is discussed and is shown to hold when $X$ is an abelian variety over $\Qb$. In this article, we give some further evidence for the validity of $GPC^k(X),$ mainly when $k=1.$ This work may be seen as a sequel of \emph{loc. cit.},   inspired by the philosophy advocated by Andr\'e in \cite{AndreMotives04}, Chapter 7, where the Grothendieck period conjecture appears as a conjecture on realization functors on categories of motives, parallel to similar ``full faithfullness conjectures'', such as the Hodge conjecture or the Tate conjecture.

Several of our results, and to some extent their proofs, may be seen as translations, in the context of the Grothendieck period conjecture, of diverse classical results concerning the Tate conjecture, that are due to Tate himself (\cite{Tate66}), Jannsen  (\cite{Jannsen90}), Ramakhrishnan and Deligne (\cite{Tate94}, (5.2) and (5.6)) and Andr\'e (\cite{Andre96}). See also \cite{Zucker77} and\cite{Charles12a} for related arguments.

Here is a short summary of some of our results, presented in an order largely unrelated to the logical organization of their proofs : 

 {\bf 1.} \emph{Stability of $GPC^1(X)$ under products.} For any two smooth projective varieties $X$ and $Y$ over $\Qb$, $GPC^1(X \times Y)$ holds iff $GPC^1(X)$ and $GPC^1(Y)$ hold.

{\bf 2.} \emph{Reduction to surfaces.} Let $X$ be a smooth projective subvariety of $\PP^N_{\Qb}$ of dimension $\geq 3$. For any linear subspace $L$ of codimension $\dim X -2$ in $\PP^N_{\Qb}$ that is transverse\footnote{namely, such that $X$ and $L$ meet properly and their scheme theoretic intersection $X \cap L$ is smooth.} to $X$, the validity of $GPC^1(X\cap L)$ implies the validity  of $GPC^1(X).$

For any smooth projective $X$ as above,  such transverse linear subspaces $L$ do exist by the theorem of Bertini, and consequently the validity of $GPC^1(X)$ for arbitrary smooth projective varieties follows from its validity for smooth projective surfaces.

{\bf 3.} \emph{Extension to open varieties. Compatibility with rational maps.} The definition of the algebraic de Rham cohomology and the construction of the comparison isomorphism (\ref{compdRC}) may be extended to an arbitrary smooth variety $X$ over $\Qb$ (\cf \cite{Grothendieck66}). As a consequence, the Grothendieck period conjecture extends as well.   



For cycles of codimension $1$, this does not lead to an actual generalization of the Grothendieck period conjecture for smooth projective varieties. Indeed we shall prove that for any smooth projective variety $X$ over $\Qb$ and any non-empty open $U$ subscheme of $X$, \emph{$GPC^1(U)$ holds iff $GPC^1(X)$ holds.}

This immediately implies the birational invariance of $GPC^1(X)$.
More generally, we shall show that, for any two smooth projective varieties $X$ and $Y$ over $\Qb,$
\emph{if there exists a dominant rational map $f: X \dasharrow Y,$ then 
$GPC^1(X)$ implies 
$GPC^1(Y).$ }

{\bf 4.} $GPC^1(X)$ holds for $X$ \emph{an abelian variety \emph{or} a $K3$ surface}, or more generally, for a smooth projective \emph{hyperk\"ahler variety} with second Betti number at least $4$.

{\bf 5.} $GPC^2(X)$ holds for $X$ a \emph{smooth cubic hypersurface in $\PP^5_{\Qb}$. }

\subsection{Organization of this article} 

In Section 2, we discuss the original formulation of the Gro\-then\-dieck period conjecture, stated in terms of the torsor of periods of a smooth projective variety $X$ over $\Qb$ and of the algebraic cycles over its powers $X^n,$ and its relation with the conjectures $GPC^k(X^n).$ Our discussion may be seen as a complement of the one by Andr\'e in \cite{AndreMotives04}, 7.5.2 and 23.1, and incorporates some interesting observations by Ayoub and Gorchinsky.

In Section 3, we recall the transcendence theorems \emph{\`a la} Schneider--Lang on which the proofs of our results will rely: these theorems provide a description of morphisms of connected algebraic groups over $\Qb$ in terms of $\Qb$-linear maps between their Lie algebras that are compatible with their ``periods". From this basic result, we derive a description of biextensions by the multiplicative group $\Gm$ of abelian varieties over $\Qb$ in terms of their ``de Rham--Betti" homology groups. In turn, this implies the stability of $GPC^1$ under products, and its validity for abelian varieties. 

In substance, the derivation of the results of Section 3 involves arguments of the same nature as the ones used in the proof of $GPC^1$ for abelian varieties in \cite{Bost12}. However we believe that emphasizing the role of biextensions leads to results that are conceptually more satisfactory, and better suited to applications.

Section 4 is devoted to the natural generalization of the conjecture $GPC^k$ concerning \emph{quasi-projective} smooth varieties over $\Qb.$ In particular, we show that the validity of $GPC^1$ for such a variety and for a smooth projective compactification are equivalent.  Here again, our main tools are the transcendence theorems on algebraic groups recalled in Section 3.  
The results in this section actually establish, in small degree, the conjecture asserting that ``Grothendieck cohomology classes on smooth quasi-projective varieties over $\Qb$ live in weight zero."

Section 5 is devoted to  results on the Grothendieck period conjecture obtained by means of various constructions involving absolute Hodge cycles. In particular, we show that the general validity of $GPC^1$ would follow from the case of smooth projective surfaces. Besides, we use the classical results of Deligne in \cite{Deligne72} concerning the Kuga--Satake correspondence to derive the validity of $GPC^1$ for $K3$ surfaces and their higher dimensional generalizations starting from its validity for abelian varieties, already established in Section 3. 
Finally, we establish $GPC^2(X)$ for a smooth cubic hypersurface $X$ in $\PP^5_{\Qb}$, by using the construction of Beauville--Donagi in \cite{BeauvilleDonagi85}. 

\bigskip

We are grateful to Joseph Ayoub and  Serguey Gorchinsky for sharing their insight regarding the relationship between the Kontsevich--Zagier conjecture and full faithfulness conjectures for categories of motives. This article has also benefited from the careful reading and  suggestions of an anonymous referee, whom we warmly thank.

During the preparation of this paper, the first author has been partially supported by the project Positive of the Agence Nationale de la Recherche (grant ANR-2010-BLAN-0119-01) and by the Institut Universitaire de France. Most of this work has been completed while the second author was a member of IRMAR at the University of Rennes 1.


\section{The Grothendieck period conjecture and the torsor of periods}\label{Period}

In this section, we discuss the relationship between the Grothendieck period conjecture and the better-known conjectures of Grothendieck and Kontsevich--Zagier on periods. The content of this section is certainly familiar to specialists and appears in various forms in \cite{AndreMotives04, Ayoub12, HuberMullerStach11}. 

At the expense of concision, and in order to keep in line with the general tone of the paper, we will focus on giving concrete statements rather than using exclusively the language of Tannakian categories.

\subsection{The de Rham-Betti category and the torsor of periods}

In this section, we unwind standard definitions in the case of the Tannakian category of de Rham-Betti realizations, see for instance \cite{DeligneMilneOgusShih}, chapter II.

\subsubsection{The categories $\CdRBQ$ and $\CdRB$.} As in \cite{Bost12} 5.3 and 5.4,  we shall use the formalism of the category $\CdRB$ of ``de Rham-Betti realizations" \emph{\`a la} Deligne-Jannsen (\cf \cite{DeligneMilneOgusShih}, 2.6, \cite{Jannsen90} and \cite{AndreMotives04}, Section 7.5). In this paper, we will often work with rational coefficients and we introduce the corresponding category $\CdRBQ$.

By definition, an object in $\CdRBQ$ is a triple 
$$M=(M_{\dR}, M_\rB, c_M),$$
where $M_{\dR}$ (resp. $M_{\rB}$) is a finite-dimensional vector space over $\Qb$ (resp. $\Q$), and $c_M$ is an isomorphism of complex vector spaces
$$c_M : M_{\dR}\otimes_{\Qb}\C\lrasim M_{\rB}\otimes_{\Q}\C.$$
For obvious reasons, the vector space $M_\dR$ (resp. $M_\rB$) is called the de Rham realization (resp. the Betti realization) of $M$. The isomorphism $c_M$ will be referred to as the comparison isomorphism.

Given two objects $M$ and $N$ in $\CdRBQ$, the group $\Hom_{\dRB,\Q}(M,N)$ of morphisms from $M$ to $N$ in $\CdRBQ$ is the subgroup of $\mathrm{Hom}_\Qb(M_{\dR}, N_{\dR})\oplus \mathrm{Hom}_\Q(M_{\rB}, N_{\rB})$ consisting of pairs $(\phi_{\dR}, \phi_{\rB})$ such that the following diagram is commutative : 
$$
\begin{CD}
M_{\dR}\otimes_{\Qb} \C@>{\phi_{\dR}\otimes_\Qb \,\rm{Id}_\C}>> N_{\dR}\otimes_{\Qb}\C\\
@V{c_M}VV                                                                 @V{c_N}VV\\
M_{\rB}\otimes_{\Q}\C @>{\phi_{\rB}\otimes_\Q\, \rm{Id}_\C}>> N_{\rB}\otimes_{\Q}\C.
\end{CD}
$$   

In more naive terms, an object $M$ of $\CdRBQ$ may be seen as the data of the finite-dimensional $\C$-vector space
$M_{\C} := M_{\dR}\otimes_{\Qb}\C \simeq M_{\rB}\otimes_{\Z}\C,$
together with a ``$\Qb$-form'' $M_{\dR}$ and a ``$\Q$-form'' $M_{\rB}$ of $M_{\C}$. Then, for any two objects $M$ and $N$ in $\CdRBQ$, the morphisms from $M$ to $N$ in $\CdRBQ$ may be identified with the $\C$-linear maps $\phi_\C: M_\C \ra N_\C$ which are compatible with both the $\Qb$-forms and the $\Q$-forms of $M$ and $N$.

For any $k \in \Z,$ we denote by $\Q(k)$ the object of $\CdRBQ$ defined by $\Q(k)_{\dR} :=\Qb$ and $\Q(k)_\rB = (2\pi i)^k \Q$ inside $\C$. 

An integral version  $\CdRB$ of the category $\CdRBQ$ is defined similarly: $M_{\rB}$ is now a free $\Z$-module of finite rank, $c_M$ an isomorphism from $M_{\dR}\otimes_{\Qb}\C$ onto $M_{\rB}\otimes_{\Z}\C$, and $\phi_{\rB}$ a morphism of $\Z$-modules. For any $k \in \Z,$ we denote by $\Z(k)$ the object of $\CdRB$ defined by $\Z(k)_{\dR} :=\Qb$ and $\Z(k)_\rB = (2\pi i)^k \Z$ inside $\C$. 

The category $\CdRB$ (resp. $\CdRBQ$) is endowed with a natural structure of rigid tensor category, with $\Z(0)$ (resp. $\Q(0)$) as a unit object, and with tensor products and duals defined in an obvious way in terms of tensor products and duality of $\C$, $\Qb$, and $\Z$ (resp. $\Q$)-modules.

Analogs of the groups $H^{2k}_\Gr$ appearing in the Grothendieck period conjecture above may be defined  in the setting of $\CdRB$.

\begin{definition}
Let $M=(M_{dR}, M_{\rB}, c_M)$ be an object of $\CdRB$ (resp. $\CdRBQ$). The $\Z$-module (resp. $\Q$-vector space) $M_\Gr$ is defined by
$$M_{\Gr}:=\mathrm{Hom}_{\dRB}(\Z(0), M)$$
(resp. 
$$M_{\Gr}:=\mathrm{Hom}_{\dRB,\Q}(\Q(0), M)).$$
\end{definition}

Clearly, the space $M_{\Gr}$ can be identified with the intersection of $M_{\rB}$ and $c_M(M_{\dR})$ inside $M_{\rB}\otimes\C$.

\subsubsection{The torsor of periods of an element of $\CdRB$.} We briefly recall the notion of an abstract torsor -- defined without specifying a structure group. We refer to \cite{HuberMullerStach11} for sorites on abstract torsors. 

If $M=(M_{\dR}, M_{\rB}, c_M)$ be an object of $\CdRBQ$, we denote by $\mathrm{Iso}(M_{\dR}\otimes_{\Qb}\C, M_{\rB}\otimes_{\Q}\C)$ the complex variety of $\C$-linear isomorphisms from $M_{\dR}\otimes_{\Qb}\C$ to $M_{\rB}\otimes_{\Q}\C$.

\begin{definition}
Let $M=(M_{\dR}, M_{\rB}, c_M)$ be an object of $\CdRBQ$. Let $V$ be a closed algebraic subset of $\mathrm{Iso}(M_{\dR}\otimes_{\Qb}\C, M_{\rB}\otimes_{\Q}\C)$. We say that $V$ is a \emph{torsor} if for any triple $(f, g, h)$ of points of $V$, the element 
$$f\circ g^{-1}\circ h : M_{\dR}\otimes_{\Qb}\C\lra M_{\rB}\otimes_{\Q}\C$$
belongs to $V$.

We say that $V$ is \emph{defined over $\Qb$} if it may be obtained by field extension from some closed algebraic subset of the variety over $\Qb$ defined as the space of $\Qb$-linear isomorphisms $\mathrm{Iso}(M_{\dR}, M_{\rB}\otimes_{\Q}\Qb)$.
\end{definition}

As follows from the above definition, an intersection of torsors is again a torsor. As a consequence, we can consider the torsor generated by a subset of $\mathrm{Iso}(M_{\dR}\otimes_{\Qb}\C, M_{\rB}\otimes_{\Q}\C)$.

\begin{definition}
Let $M=(M_{\dR}, M_{\rB}, c_M)$ be an object of $\CdRBQ$. The \emph{torsor of periods} of $M$, which we denote by $\Omega_M$, is the torsor generated by the Zariski closure $Z_M$  of $c_M$ in the $\Qb$-scheme $\mathrm{Iso}(M_{\dR}, M_{\rB}\otimes_{\Q}\Qb)$. 
\end{definition}
By definition, $Z_M(\C)$ is the intersection of all $\Qb$-algebraic subsets of $\mathrm{Iso}(M_{\dR}\otimes_{\Qb}\C, M_{\rB}\otimes_{\Q}\C)$ that contain $c_M$.

At this level of generality, it is not easy to describe concretely the torsor of periods of a given object of $\CdRB$. However, Grothendieck classes provide equations for this torsor as follows.

Let $M=(M_{\dR}, M_{\rB}, c_M)$ be an object of $\CdRBQ$. Let $m$ and $n$ be two nonnegative integers, and let $k$ be an integer. Any isomorphism 
$$ f : M_{\dR}\lra M_{\rB}\otimes_\Q \Qb$$
induces a canonical isomorphism from $(M^{\otimes m}\otimes (M^{\vee})^{\otimes n}\otimes \Q(k))_\dR$ to $(M^{\otimes m}\otimes (M^{\vee})^{\otimes n}\otimes \Q(k))_\rB\otimes_\Q\Qb$. We will denote it by $f$ as well.

\begin{definition}\label{torsor}
Let $M=(M_{\dR}, M_{\rB}, c_M)$ be an object of $\CdRBQ$.

Given an element $\alpha$ in $(M^{\otimes m}\otimes (M^{\vee})^{\otimes n})_\Gr$, let $\Omega_{\alpha}$ be the torsor whose $\Qb$-points are the isomorphisms 
$$ f : M_{\dR}\lra M_{\rB}\otimes_\Q \Qb$$
such that 
$$f(\alpha_\dR)=\alpha_\rB.$$
The \emph{Tannakian torsor of periods} of $M$, which we denote by $\Omega_M^T$, is the intersection of the $\Omega_\alpha$ as $m, n$ and $\alpha$ vary. 
\end{definition}

By definition of Grothendieck classes, $\Omega_\alpha$ is defined over $\Qb$. Tautologically, since $$\alpha_\rB=c_{M^{\otimes m}\otimes (M^{\vee})^{\otimes n}}(\alpha_\dR),$$ the comparison isomorphism $c_M$ is a complex point of $\Omega_{\alpha}$. The lemma below follows.

\begin{lemma}
Let $M=(M_{\dR}, M_{\rB}, c_M)$ be an object of $\CdRBQ$. Then
$$\Omega_M\subset\Omega_M^T.$$
\end{lemma}

\bigskip

The discussion above can be readily rephrased in a more concise way, using the fact that the category $\CdRBQ$ is a Tannakian category. Namely, both $M\mapsto M_\dR$ and $M\mapsto M_\rB\otimes_\Q \Qb$ are fiber functors with value in the category of $\Qb$-vector spaces. Those are the de Rham and the Betti realization of $\CdRBQ$, respectively. Isomorphisms between these two fiber functors give rise to a torsor under the Tannakian group of $\CdRBQ$. Now any object $M$ in $\CdRBQ$ gives rise to a Tannakian subcategory $\langle M\rangle$ generated by $M$. The torsor of isomorphisms between the de Rham and the Betti realization of $\langle M\rangle$ is precisely $\Omega_M^T$ -- hence the notation. It is a torsor under the Tannakian fundamental group of $M$ -- more precisely, this fundamental group may be realized as a $\Q$-subgroup $G$ of $GL(M_\rB)$, and $\Omega^T_M$ is a torsor under $G_{\Qb}$.

\begin{remark}\label{strict-inclusion}
In general, the inclusion of $\Omega_M$ in $\Omega^T_M$ is strict. Indeed, $\Omega_M$ is a torsor under a subgroup $H$ of $GL(M_\rB\otimes\Qb)$. If $\Omega^T_M=\Omega_M$, then the group $H$ would be equal to the group $G_{\Qb}$ above. In particular, it would be defined over $\Q$. However, it is easy to construct an object $M$ in $\CdRBQ$, with $\dim M_B=\dim M_\dR = 2$, such that the group $H\subset GL(M_\rB\otimes\Qb)$ above is not defined over $\Q$.
\end{remark}

\begin{remark}\label{lefschetz}
Let $M$ be an element of $\CdRBQ$. If $\Q(1)$ is an object of $\langle M\rangle$, for any $(m, n, k)\in\N^2\times\Z$, and any element $\alpha\in (M^{\otimes m}\otimes (M^{\vee})^{\otimes n}\otimes\Q(k))_\Gr$, the Tannakian torsor $\Omega_M^T$ is contained in $\Omega_\alpha$, where $\Omega_\alpha$ is defined by the obvious extension of Definition \ref{torsor}.
\end{remark}

\subsection{The Zariski closure of the torsor of periods and transcendence conjectures}
After the general discussion above, we specialize to the case of objects in $\CdRBQ$ coming from the cohomology of algebraic varieties.

\subsubsection{Torsor of periods and de Rham-Betti realization}\label{realization}

Let $X$ be a smooth projective variety over $\Qb$. As explained in the introduction, given a nonnegative integer $k$ and an integer $j$, the comparison isomorphism between de Rham and Betti cohomology allows us to associate to $X$ an object $H^k_\dRB(X, \Z(j))$ in $\CdRB$, its $k$-th de Rham-Betti cohomology group\footnote{Note that, by definition, the $\Z$-modules appearing in objects of $\CdRB$ are torsion-free. Accordingly, whenever Betti homology or cohomology groups with integer coefficients appear, it will be understood that these are considered modulo their torsion subgroup.} with coefficients in $\Z(j)$, as well as its rational version $H^k_\dRB(X, \Q(j))$ in $\CdRBQ$. Moreover, the compatibility of the cycle maps with the comparison isomorphism between de Rham and Betti cohomology induces a cycle map 
$$cl^X_\Gr : Z^k(X)\ra H^{2k}_\dRB(X, \Q(k))_\Gr.$$
Of course, $H^{2k}_\dRB(X, \Q(k))_\Gr=H^{2k}_\Gr(X, \Q(k))$ and this map coincides with the one introduced in \ref{statement}. If $k=1$, the map factorizes throughout $\Pic(X)$ and defines a map 
$$c_{1,\Gr}^X : \Pic(X)\ra H^{2}_\Gr(X, \Q(1)).$$ 
For any integers $k$ and $j$, we will write $H^k_\Gr(X, \Q(j))$ (resp. $H^k_\Gr(X, \Z(j))$) for $H^k_\dRB(X, \Q(j))_\Gr$ (resp. $H^k_\dRB(X, \Z(j))_\Gr$).

\bigskip

The de Rham (resp. Betti) realization of $H^k_\dRB(X, \Z(j))$ is by definition $H^k_\dR(X/\Qb)$ (resp. $$H^k(X^\an_\C, \Z(j)):= (2i\pi)^jH^k(X^\an_\C, \Z)).$$
The comparison isomorphism is the one induced from (\ref{compdRC}) and (\ref{compBC}). The comparison isomorphism can be rewritten in terms of actual periods. Indeed, the $k$-th homology group $H_{k}(X^\an_\C, \Z)$ is dual to $H^k(X^\an_\C, \Z(j))$ via the map 
$$\gamma\longmapsto \frac{1}{(2i\pi)^j}(\gamma, .),$$
where $(., .)$ denotes the canonical pairing between homology and cohomology. In these terms, the inverse of the comparison isomorphism
$$H^k_\dR(X/\Qb)\otimes\C\lra H^k(X^\an_\C, \Z(j))\otimes\C$$
is dual to the pairing
\begin{equation}\label{period}
H^k_\dR(X/\Qb)\otimes H_{k}(X^\an_\C, \Z(0))\lra\C, \alpha\otimes\gamma\longmapsto\frac{1}{(2i\pi)^j}\int_\gamma \alpha.
\end{equation}

We denote by $H^{\bullet}_\dRB(X, \Z(0))$ the object $\bigoplus_k H^k_\dRB(X, \Z(0))$ in $\CdRB$, and by $H^{\bullet}_{\dRB}(X, \Q(0))$ its rational variant in $\CdRBQ$. The discussion of the previous paragraph applied to $M=H^{\bullet}_{\dRB}(X, \Q(0))$ gives rise to torsors naturally associated to the de Rham-Betti cohomology of $X$.

\begin{lemma}\label{factor-lefschetz}
Let $X$ be a smooth projective variety over $\Qb$. Then $\Q(-1)$ is a direct factor of $H^2_\dRB(X, \Q(0))$.
\end{lemma}

\begin{proof}
let $[H]$ be the cohomology class of a hyperplane section of $X$. This class corresponds to a map 
$\Q(0)\ra H^2_\dRB(X, \Q(1))$. By Poincar\'e duality and the hard Lefschetz theorem, the bilinear form 
$$\alpha\otimes\beta\mapsto \int_X\alpha\cup\beta\cup[H]^{\mathrm{dim}(X)-2}$$
is non-degenerate both on $H^2(X^\an_\C, \Q(0))$ and $H^2_\dR(X/\Qb)$. 

Since it is compatible to the comparison isomorphism -- as the latter is compatible with the algebra structure on cohomology and the trace map -- the orthogonal of $[H]$ in both $H^2_\rB(X, \Q)$ and $H^2_\dR(X/\Qb)$ corresponds to a subobject of $H^2_\dRB(X, \Q(1))$. Since $[H]^{\mathrm{dim}(X)}\neq 0$, this shows that $\Q.[H]$ is a direct factor of $H^2_\dRB(X, \Q(1))$, isomorphic to $\Q(0)$. As a consequence, $\Q(-1)$ is a direct factor of $H^2_\dRB(X, \Q(0))$. 
\end{proof}

\begin{definition}\label{torsorsdef}
Let $X$ be a smooth projective variety over $\Qb$.
\begin{enumerate}
\item The \emph{torsor of periods} of $X$, which we denote by $\Omega_X$, is the torsor of periods of $H^{\bullet}_{\dRB}(X, \Q(0))$, that is, the torsor generated by the Zariski-closure $Z_X:=Z_{H^{\bullet}_{\dRB}(X, \Q(0))}$  of $c_{H^{\bullet}_{\dRB}(X, \Q(0))}$ in the $\Qb$-scheme $\mathrm{Iso}(H^{\bullet}_{\dR}(X/\Qb), H^{\bullet}_{\rB}(X, \Q)\otimes_{\Q}\Qb)$.
\item The  \emph{Tannakian torsor of periods} of $X$, which we denote by $\Omega_X^T$, is the Tannakian torsor of periods of $H^{\bullet}_{\dRB}(X, \Q(0))$.
\item The  \emph{torsor of motivated periods} of $X$, which we denote by $\Omega_X^{And}$, is the intersection of the torsors $\Omega_{\alpha}$ defined in definition \ref{torsor}, where $\alpha$ runs through cycle classes of motivated cycles -- in the sense of Andr\'e \cite{Andre96b} -- in the de Rham-Betti realizations $H^{2k}_\dRB(X^n, \Q(k))$ as $n$ and $k$ vary.
\item The  \emph{motivic torsor of periods} of $X$, which we denote by $\Omega_X^{mot}$, is the intersection of the torsors $\Omega_{\alpha}$ defined in definition \ref{torsor}, where $\alpha$ runs through cycle classes of algebraic cycles in the de Rham-Betti realizations $H^{2k}_\dRB(X^n, \Q(k))$ as $n$ and $k$ vary.
\end{enumerate}
\end{definition}

The motivic torsor of periods $\Omega_X^{mot}$ is what is called the torsor of periods in \cite{AndreMotives04}, chapitre 23.  The cohomology of $X^n$ is a direct factor (!) of $H^{\bullet}_\dRB(X, \Q(0))^{\otimes n}$ by the K\"unneth formula. Using Lemma \ref{factor-lefschetz}, this justifies the definition of $\Omega_X^{And}$ and $\Omega_X^{mot}$. Under the standard conjectures \cite{Grothendieck68Stand}, $\Omega_X^{mot}$ is a torsor under the motivic Galois group of $X$.

\begin{lemma}\label{equations}
Let $X$ be a smooth projective variety over $\Qb$. The Tannakian torsor of periods of $X$ is the intersection of the torsors $\Omega_{\alpha}$ defined in definition \ref{torsor}, where $\alpha$ runs through Grothendieck classes in the de Rham-Betti realizations $H^{j}_\dRB(X^n, \Q(k))$ as $j$, $n$ and $k$ vary.
\end{lemma}

\begin{proof}
Lemma \ref{factor-lefschetz} and Remark \ref{lefschetz} show that $\Omega_X^T$ is the intersection of the $\Omega_{\alpha}$, as $\alpha$ runs through Grothendieck classes in tensor products of the cohomology groups of $X$, their dual and $\Q(k)$. Using Poincaré duality and the K\"unneth formula, this proves the lemma.
\end{proof}

\begin{corollary}\label{inclusion}
Let $X$ be a smooth projective variety over $\Qb$. We have
\begin{equation}\label{inclus}
Z_X\subset \Omega_X\subset \Omega_X^T\subset \Omega_X^{And}\subset \Omega_X^{mot}.
\end{equation}
\end{corollary}

\subsubsection{Transcendence and full faithfulness conjectures for smooth projective varieties}\label{conjectures}

The Grothen\-dieck Period  Conjecture of \cite{Grothendieck66} (note (10), p. 102) is the following.

\begin{conjecture}\label{Grothendieck}
Let $X$ be a smooth projective variety over $\Qb$. Then 
$$Z_X=\Omega_X^{mot}.$$
In other words, the comparison isomorphism is dense in the motivic torsor of periods.
\end{conjecture}

Given Corollary \ref{inclusion}, Conjecture \ref{Grothendieck} would imply that all of the inclusions in (\ref{inclus}) are equalities. As in \cite[4.2]{KontsevichZagier01}, it has a simple interpretation in terms of periods, meaning that any polynomial relation between periods of the form  
$$\frac{1}{(2i\pi)^j}\int_\gamma \alpha,$$
where $j$ is any integer and $\alpha$ (resp. $\gamma$) is an element of $H^k_\dR(X^n/\Qb)$ (resp. $H_{k}((X^n)^\an_\C, \Q)$) for some nonnegative $n$, is induced by algebraic cycles on self-products of $X$.

There are few cases where Conjecture \ref{Grothendieck} is known, the most significant one being perhaps the case where $X$ is an elliptic curve with complex multiplication, due to Chudnovsky \cite{Chudnovsky78}. 

Our next result relates the conjectures $GPC^k$ to the inclusions (\ref{inclus}).

\begin{proposition}\label{relation}
Let $X$ be a smooth projective variety over $\Qb$. 
\begin{enumerate}
\item Assume that $GPC^k(X^n)$ holds for every $k$ and $n$, and that $H^j_{\Gr}(X^n, \Q(k))=0$ unless $j=2k$. Then $\Omega^T_X=\Omega^{mot}_X$.
\item Assume that $X$ satisfies the standard conjectures of \cite{Grothendieck68Stand} and that $\Omega^T_X=\Omega^{mot}_X$. Then $GPC^k(X^n)$ holds for every $k$ and $n$, and $H^j_{\Gr}(X^n, \Q(k))=0$ unless $j=2k$.
\end{enumerate}
\end{proposition}

\begin{proof}
First assume that $GPC^k(X^n)$ holds for every $k$ and $n$, and that $H^j_{\Gr}(X^n, \Q(k))=0$ unless $j=2k$. Lemma \ref{equations} then shows that $\Omega_X^T=\Omega_X^{mot}$, as they are defined by the same equations. 

Now assume that $X$ satisfies the standard conjectures of \cite{Grothendieck68Stand} and that $\Omega^T_X=\Omega^{mot}_X$. This implies that the motivic Galois group $G_{mot}(X)$ of $X$ -- with respect to the Betti realization -- is a well-defined reductive group over $\Q$, coming from the Tannakian category of pure motives generated by $X$, and that $\Omega_X^T$ is a torsor under $G_{mot}(X)_{\Qb}$.

Let $\alpha$ be an element of $H^j_{\dRB}(X^n, \Q(k))_\Gr$ for some $j,k$ and $n$. By definition of $\Omega_X^T$, and since $\Omega_X^{mot}=\Omega_X^T$, if $f$ is any point of $\Omega_X^{mot}$, $f(\alpha_\rB)=c_X(\alpha_\dR)$, where $c_X$ is the comparison isomorphism. Deligne's principle A of \cite{DeligneMilneOgusShih}, or rather its Tannakian proof as in \cite{Blasius94}, 2.11, implies that $\alpha$ is the cohomology class of an algebraic cycle. In particular, $j=2k$, which proves the proposition.
\end{proof}

The same proof gives the following results for motivated cycles.

\begin{proposition}\label{relationmot}
Let $X$ be a smooth projective variety over $\Qb$. 
\begin{enumerate}
\item Assume that for every $k$ and $n$, classes in $H^{2k}_\Gr(X^n, \Q(k))$ are classes of motivated cycles, and that $H^j_{\Gr}(X^n, \Q(k))=0$ unless $j=2k$. Then $\Omega^T_X=\Omega^{And}_X$. 
\item Assume that $\Omega^T_X=\Omega^{And}_X$. Then for every $k$ and $n$, Grothendieck classes in $H^{2k}_\dRB(X^n, \Q(k))$ are classes of motivated cycles, and $H^j_{\Gr}(X^n, \Q(k))=0$ unless $j=2k$.
\end{enumerate}
\end{proposition}


The results above explains in which respect the conjectures $GPC^k$ are weaker than Conjecture \ref{Grothendieck}. Indeed, they do not address whether the Zariski-closure $Z_X$ of the comparison isomorphism is actually a torsor. We have nothing to say in this direction -- see the recent work of Ayoub \cite{Ayoub12} for related results in the function field case. 

In more concrete terms, this corresponds to the fact that while Conjecture \ref{Grothendieck} addresses the transcendence of any single period $\frac{1}{(2i\pi)^j}\int_\gamma \alpha$, the conjectures $GPC^k$ deal with the existence, given a de Rham cohomology class $\alpha$, of some element $\gamma$ of Betti cohomology such that $\frac{1}{(2i\pi)^j}\int_\gamma \alpha$ is transcendental.

Additionally, it should be noted that the torsor $\Omega_X$ only depends on the triple 
$$(H^{\bullet}_\dR(X/\Qb), H^{\bullet}_\rB(X, \Q)\otimes \Qb, H^{\bullet}_\dR(X/\Qb)\otimes\C\ra H^{\bullet}_\rB(X, \Q)\otimes\C),$$
and as such does not depend on the $\Q$-structure of $H^{\bullet}_\rB(X, \Q)\otimes \Qb$, whereas $\Omega^T_X$ a priori does -- see Remark \ref{strict-inclusion}.


\bigskip

Propositions \ref{relation} and \ref{relationmot} also show that the conjectures $GPC^k$ should be supplemented by the conjectures asserting that, if $X$ is a smooth projective variety over $\Qb$, then $H^j_{\Gr}(X, \Q(k))=0$ unless $j=2k$. 

For general $j$ and $k$, this conjecture seems widely open, and corresponds to the lack of a theory of weights for the de Rham-Betti realization of the cohomology of smooth projective varieties over $\Qb$. We will discuss this issue in \ref{wfiltration}.

\subsubsection{A few remarks about the mixed case}

Most of the discussion and the conjectures above could be extended to the framework of  arbitrary varieties over $\Qb$, without smoothness or projectivity assumptions. The de Rham-Betti realization still makes sense, as well as the notion of Grothendieck classes, as we recall at the beginning of section 4. It is possible, with some care, to state conjectures similar to $GPC^k$ in this setting.

As in the previous paragraph, the Kontsevich--Zagier conjecture of \cite[Section 4]{KontsevichZagier01} bears a similar relationship to the conjectures $GPC^k$ in the mixed case as Conjecture \ref{Grothendieck} does in the pure case. As the results of our paper mostly deal with the pure case, we will not delve in this theoretical setting any further. Let us however give one result in that direction -- another one for open varieties will be discussed below in Section 4.

Observe that, for any given smooth variety $X$ over $\Qb$, there exists a cycle map from the higher Chow groups $CH^i(X, n)$ to the $\Q$-vector space $H^{2i-n}_\Gr(X, \Q(i))$ of Grothendieck classes in the de Rham--Betti group $H^{2i-n}_\dRB(X, \Q(i))$. As in the usual case of Chow groups, this is due to the compatibility of the cycle maps to the Betti and de Rham cohomology; see for instance \cite{Jannsen90}.

\begin{theorem}
For any smooth quasi-projective variety $U$ over $\Qb$, the cycle map 
$$CH^1(U, 1)_\Q\lra H^1_\Gr(U, \Q(1))$$
is surjective.
\end{theorem}

\begin{proof} We only give a sketch of the proof and leave the details to the reader.

Using resolution of singularities, we can find a smooth projective variety $X$ over $\Qb$ containing $U$ such that the complement of $U$ in $X$ is a divisor $D$. \emph{Mutatis mutandis}, the arguments in \cite[Corollary 9.10]{Jannsen90} show that the cycle map $CH^1(U, 1)_\Q\lra H^1_\Gr(U, \Q(1))$ is surjective if and only if the Abel-Jacobi map in de Rham-Betti cohomology
$$\Pic^0(X)\otimes\Q\lra \mathrm{Ext}^1_\dRB(\Q(0), H^1_\dRB(X, \Q(1)))$$
is injective. This  Abel-Jacobi map coincides with the map $\kappa_{\dRB}$ attached to the Albanese variety $A$ of $X$ that is defined in \cite{Bost12}, Section 5.5. As observed in \cite{Bost12}, Proposition 5.4,  its injectivity is a consequence of Theorem \ref{SLmor} \emph{infra}, applied to $G_1 = A$ and to $G_2$ an extension of $A$ by $\Gm_\Qb.$ \end{proof}

\section{Transcendence and de Rham-Betti cohomology of abelian varieties. Applications to biextensions and divisorial correspondences}

\subsection{Transcendence and periods of commutative algebraic groups over $\Qb$}\label{TransSL}

For any smooth algebraic group over some field $k$, we denote 
$$\Lie G := T_e G$$
its Lie algebra -- a $k$-vector space of rank $\dim G$. A $k$-morphism $\phi : G_1\ra G_2$ of smooth algebraic groups over $k$ induces a $k$-linear map 
$$\Lie \phi : = D\phi(e) : \Lie G_1\ra \Lie G_2$$
between their Lie algebras. This construction is clearly compatible with extensions of the base field $k$.

Let $G$ be a connected commutative algebraic group over $\C$. Its analytification $G^\an$ is a connected commutative complex Lie group. The exponential map $\exp_G$ of this Lie group is an \'etale, hence surjective, morphism of complex Lie groups from the vector group $\Lie G$ defined by the Lie algebra of $G$  to this analytification $G^\an.$ The kernel of $\exp_G$
$$\Per G := \Ker \exp_G$$
--- the group of ``periods" of $G$ --- is a discrete subgroup of $\Lie G$, and fits into an exact sequence of commutative complex Lie groups:
$$0 \lra \Per{G}\, \hlra \Lie G \xrightarrow{\exp_{G}} G^\an \lra 0.$$ 

Let $G_1$ and $G_2$ be two connected commutative algebraic groups over $\Qb$. Consider an element $\phi$ in the $\Z$-module $\Hom_{\Qb-\ggp}(G_{1}, G_{2})$ of morphisms of algebraic groups over $\Qb$ from $G_1$ to $G_2$. This $\Qb$-linear map
$$\Lie \phi := D\phi (e) : \Lie G_{1} \lra \Lie G_{2}$$
is compatible with the exponential maps of $G_{1,\C}$ and $G_{2,\C}$, in the sense that the $\C$-linear map $\Lie \phi_\C=(\Lie\phi)_\C$ fits into a commutative diagram
 $$
 \begin{CD}
 \Lie G_{1\C} @>{\Lie \phi_{\C}}>> \Lie G_{2\C} \\
@V{\exp_{G_{1\C}}}VV                @VV{\exp_{G_{2\C}}}V \\
G_{1\C}^\an @>{\phi_{\C}}>>       G_{2\C}^\an
\end{CD}.
$$
In particular,
$$(\Lie \phi)_{\C}(\Per G_{1\C}) \subset \Per G_{2\C}.$$

This construction defines an injective morphism of $\Z$-modules:
\begin{equation}\label{LieGamma}
\Lie : \Hom_{\Qb-\ggp}(G_{1}, G_{2}) \lra \{ \psi \in \Hom_{\Qb}(\Lie G_{1},\Lie G_{2})\vert \psi_{\C}(\Per G_{1\C}) \subset \Per G_{2\C} \}.
\end{equation}

In the next sections, we shall use the following description of the morphisms of connected commutative algebraic groups over $\Qb$ in terms of the associated morphisms of Lie algebras and period groups:

\begin{theorem}\label{SLmor}
If the group of periods $\Per G_{1\C}$ generates $\Lie G_{1\C}$ as a complex vector space, then the map (\ref{LieGamma}) is an isomorphism of $\Z$-modules.
\end{theorem}

This theorem is a consequence of the classical transcendence theorems \emph{\`a la} Schneider--Lang (\cite{Schneider41}, \cite{Lang65}, \cite{Waldschmidt87}). See \cite{Bertrand83}, Section 5, Prop. B, and \cite{Bost12}, Cor. 4.3.

When $G_1$ is the multiplicative group $\Gm_{,\Qb}$, then $\Lie G_1$ is a one dimensional $\Qb$-vector space, with basis the invariant vector field $ X\frac{\partial}{\partial X}$, and the group of periods $\Per G_{1\C}$ is the subgroup $2\pi i\Z X\frac{\partial}{\partial X}$ of $\C X\frac{\partial}{\partial X}$. The hypothesis of Theorem \ref{SLmor} is then satisfied, and we obtain:

\begin{corollary}\label{SLmorGm} For any connected commutative algebraic group $G$ over $\Qb,$ we have an isomorphism of $\Z$-modules:
$$
\begin{array}{rcl}
\Hom_{\Qb-\ggp}(\Gm_\Qb, G)  & \stackrel{\sim}{\lra}   &\{v \in \Lie G \mid 2\pi i v \in \Per G_\C \} = \Lie G  \cap \frac{1}{2\pi i} \Per G_\C   \\ \phi
  & \longmapsto   & \Lie \phi (X\frac{\partial}{\partial X}).   
\end{array}
$$
\end{corollary}

We finally recall that the theorem of Schneider-Lang also provides a Lie theoretic description of morphisms of $\Qb$-algebraic groups of source the additive group $\Ga_{,\Qb}$; see for instance \cite{Bost12}, Theorem 4.2:

\begin{theorem}\label{SLmorGa} For any connected commutative algebraic group $G$ over $\Qb,$ we have an isomorphism of $\Z$-modules\footnote{This still holds, as a bijection of sets, when $G$ is an arbitrary algebraic group over $\Qb$.}:
$$
\begin{array}{rcl}
\Hom_{\Qb-\ggp}(\Ga_\Qb, G)  & \stackrel{\sim}{\lra}   &\{v \in \Lie G \mid \exp_{G_\C} (\C v) \cap G(\Qb) \neq \emptyset \}    \\ \phi
  & \longmapsto   & \Lie \phi (\frac{\partial}{\partial X}).   
\end{array}
$$
\end{theorem}

As any morphism in $\Hom_{\Qb-\ggp}(\Ga_\Qb, G)$ is either zero or injective, this immediately yields: 

\begin{corollary}\label{corSLmorGa} For any connected commutative algebraic group $G$ over $\Qb,$ we have:
$$ \Lie G  \cap  \Per G_\C = \{0\}.$$
\end{corollary}

\subsection{Divisorial correspondences and biextensions of abelian varieties}\label{DCBiextAb}

In this section, we gather diverse basic facts concerning divisorial correspondences between smooth projective varieties and biextensions of abelian varieties. We state them in the specific framework of varieties over $\Qb$, where they will be used in this article, although, suitably formulated, they still hold over an arbitrary base. For proofs and more general versions, we refer the reader to \cite{Lang59} Chapter VI, \cite{Raynaud70} (notably Chapters III, IV, and XI),  \cite{VarAbOrsay}, and \cite{SGA7I} Expos\'es VII and VIII (notably VII.2.9 and VIII.4).

\subsubsection{Notation}
Let $X$ be a smooth projective variety over $\Qb$, equipped with some ``base point" $x \in X(\Qb).$ To $X$ is attached its Picard group $\Pic  (X) := H^1(X,\cO_X^\times)$, its connected Picard variety $\Pic^0_{X/\Qb}$ (the abelian variety that classifies line bundles over $X$ algebraically equivalent to zero), and its N\'eron-Severi group
$$NS(X) := \Pic(X) / \Pic^0_{X/\Qb}(\Qb),$$
that is, the group of line bundles over $X$ up to algebraic equivalence. 

We shall also consider the Albanese variety of $X$, defined as the abelian variety 
$$\Alb (X) := (\Pic^0_{X/\Qb})^\wedge$$
dual to  $\Pic^0_{X/\Qb}$, and the Albanese morphism
$$\alb_{X,x} : X \lra \Alb(X).$$
It is characterized by the fact that the pullback by $(\alb_{X,x}, \rm{Id}_{\Pic^0_{X/\Qb}})$ of a Poincar\'e bundle on $(\Pic^0_{X/\Qb})^\wedge \times 
\Pic^0_{X/\Qb}$ is isomorphic to a Poincar\'e bundle over $X \times \Pic^0_{X/\Qb}$ (trivialized along $\{X\} \times \Pic^0_{X/\Qb}$). It is also a ``universal pointed morphism" from $(X,x)$ to an abelian variety.

\subsubsection{Divisorial correspondences}

Let $X$ and $Y$ be two smooth projective varieties over $\Qb$, equipped with base points $x \in X(\Qb)$ and $y\in Y(\Qb).$

The group of divisorial correspondences $DC(X,Y)$ between $X$ and $Y$ may be defined as a subgroup of $\Pic(X\times Y)$ by the following condition, for any line bundle $L$ over $X\times Y$ of class $[L]$ in  $\Pic(X\times Y)$:
$$[L] \in DC(X,Y) \Longleftrightarrow L_{\mid X\times \{y\}} \simeq \cO_{X\times \{y\}} \mbox{ and }  L_{\mid \{x\}\times Y} \simeq \cO_{\{x\}\times Y}.$$

This construction is clearly functorial in $(X, x)$ and $(Y, y)$: if $(X',x')$ and $(Y',y')$ are two pointed smooth projective varieties over $\Qb$ and if 
$f : X' \lra X$ and $g: Y' \lra Y$
are two $\Qb$-morphisms such that $f(x') =x$ and $f(y') =y$, then the pullback morphism
$$(f,g)^\ast : \Pic(X \times Y) \lra \Pic(X' \times Y')$$
defines, by restriction, a morphism of abelian groups:
$$(f,g)^\ast: DC(X,Y) \lra DC(X',Y').$$

Observe that $\Pic(X)$ and $\Pic(Y)$ may be identified with subgroups of $\Pic(X \times Y)$ (by means of the pullback by the projections from $X \times Y$ to $X$ and $Y$), and that, taking these identifications into account, we get a functorial decomposition of the Picard group of the product $X \times Y$:
 \begin{equation}\label{PicXY}
\Pic(X \times Y)   \lrasim \Pic(X) \oplus \Pic(Y) \oplus DC(X,Y).
 \end{equation}
 
 Moreover the Picard variety $\Pic^0_{X\times Y/\Qb}$ may be identified with $\Pic^0_{X/\Qb}\times\Pic^0_{Y/\Qb}$, and consequently the subgroup $\Pic^0_{X\times Y/\Qb}(\Qb)$ of $\Pic(X\times Y)$ with the product of the subgroups  $\Pic^0_{X/\Qb}(\Qb)$ and  $\Pic^0_{Y/\Qb}(\Qb)$ of $\Pic(X)$ and $\Pic(Y)$. The composite map 
 $$DC(X,Y) \hlra \Pic (X \times Y) \twoheadrightarrow NS(X\times Y)$$
 is therefore injective, and, if we still denote $DC(X,Y)$ its image in $NS(X\times Y)$, the decomposition (\ref{PicXY}) becomes, after quotienting by $\Pic^0_{X\times Y/\Qb}(\Qb)$ :
\begin{equation}\label{NSXY}
NS(X \times Y)   \lrasim NS(X) \oplus NS(Y) \oplus DC(X,Y).
 \end{equation}
 
Also observe that, through the cycle maps, the decompositions (\ref{PicXY}) and (\ref{NSXY}) are compatible with the K\"unneth decomposition of the second cohomology group of $X \times Y$.

 \subsubsection{Divisorial correspondences and biextensions of abelian varieties}
 
 The next two propositions show that the group $DC(X,Y)$ of divisorial correspondences associated to some smooth projective varieties (over $\Qb$) may be identified with the group $ \Biext^1_{\Qb-\ggp}(\Alb(X), \Alb(Y);\Gm)$ of biextensions of their Albanese varieties by the multiplicative group $\Gm$.

 \begin{proposition}\label{prop:DCAlb} For any two smooth projective varieties $X$ and $ Y$ over $\Qb,$ equipped with base points $x\in X(\Qb)$ and $y \in Y(\Qb),$ the Albanese morphisms $\alb_{X,x}$ and $\alb_{Y,y}$ induce isomorphisms of groups of divisorial correspondences:
  \begin{equation}\label{DCAlb}
(\alb_{X,x}, \alb_{Y,y})^\ast : DC(\Alb(X), \Alb(Y)) \lrasim DC(X,Y).
 \end{equation}
\end{proposition}

Let $A_1$ and $A_2$ be two abelian varieties over $\Qb.$ 

Recall that a biextension of $(A_1,A_2)$ by $\Gm$ (over $\Qb$) is a $\Gm$-torsor over $A_1 \times A_2$ equipped with two compatible partial group laws. In particular, as a $\Gm$-torsor, it is trivialized over $A_1 \times \{0\}$ and $\{0\} \times A_2$, hence defines an element of $DC(A_1,A_2).$

 In turn, if $L$ is  a line bundle over $A_1 \times A_2$ trivialized over $\{0\} \times A_2$, then, for any $x \in A_1(\Qb),$ the line bundle $L_{\mid \{x\} \times A_2}$ is algebraically equivalent to zero and therefore defines a $\Qb$-point $\alpha_L(x)$ of the dual abelian variety $A_2^\wedge$. Moreover  this construction defines a morphism of $\Qb$-algebraic groups
 $$\alpha_L: A_1 \lra A_2^\wedge.$$
 
 If we switch the roles of $A_1$ and $A_2$ in this discussion, we get the morphism of abelian varieties dual to the previous one:  
 $$\alpha_L^\wedge: A_2 \lra A_1^\wedge.$$

 \begin{proposition}\label{prop:DCBiextab}
For any two abelian varieties $A_1$ and $A_2$ over $\Qb,$ the above constructions define isomorphisms of $\Z$-modules:
 \begin{equation}\label{DCBiextab}  \Biext^1_{\Qb-\ggp}(A_1,A_2;\Gm)  \lrasim DC(A_1, A_2) \lrasim \Hom_{\Qb-\ggp}(A_1,A_2^\wedge) \lrasim \Hom_{\Qb-\ggp}(A_2,A_1^\wedge).  
 \end{equation}
 \end{proposition}
 
 We finally recall the description of the N\'eron-Severi group of an abelian variety in terms of its symmetric biextensions by $\Gm.$
 
 Let $A$ be an abelian variety over $\Qb$, and let
$m, \pr_1, \pr_2: A \times A \lra A$
 denote respectively the addition law and the two projections. According to the theorem of the cube, for any line bundle $L$ over $A,$ the line bundle
 $$\Lambda (L) := m^\ast L \otimes \pr_1^\ast L^\vee \otimes \pr_2^\ast L^\vee \otimes L_0$$
 --- or rather the corresponding $\G_m$ torsor over $A\times A$ --- is equipped with a canonical structure of symmetric biextension  of $(A,A)$ by $\Gm.$ Moreover, according to the theorem of the square, the class of $\Lambda(L)$ in the subgroup $\SymBiext^1_{\Qb-\ggp}(A,A;\Gm)$ of symmetric biextensions in $\Biext^1_{\Qb-\ggp}(A,A;\Gm)$ depends only on the class of $L$ in the N\'eron-Severi group of $A$.
 
  \begin{proposition}\label{prop:SymBiextab}
For any abelian variety $A$ over $\Qb,$ the above construction, together with the isomorphisms in Proposition  \ref{prop:DCBiextab} with $A_1= A_2=A$, define isomorphisms of $\Z$-modules:
 \begin{equation}\label{SymBiextab}  NS(A) \lrasim \SymBiext^1_{\Qb-\ggp}(A, A;\Gm) ) \lrasim \Hom_{\Qb-\ggp}^{\rm sym}(A, A^\wedge) := \{ \phi \in  \Hom_{\Qb-\ggp}(A, A^\wedge) \mid \phi^\wedge = \phi \}.  
 \end{equation}
 \end{proposition}

 \subsection{Transcendence and de Rham-Betti (co)homology groups of abelian varieties. Application to biextensions}  
 
 In this section, we combine the transcendence results of Section \ref{TransSL} and the relations between N\'eron-Severi groups, divisorial correspondences, and biextensions recalled in Section \ref{DCBiextAb} to derive diverse full faithfulness properties of the de Rham-Betti realization. These results constitute variants and complements of the results in \cite{Bost12}, Sections 5.2-4, that we now briefly recall. 
 
 To any abelian variety $A$ over $\Qb$ is attached its de Rham-Betti cohomology group 
 $$H^1_{\dRB}(A) := H^1_{\dRB}(A, \Z(0)),$$
 and its de Rham-Betti homology group
 $$H_{1,\dRB}(A) := H^1_{\dRB}(A)^\vee,$$
the object in $\CdRB$ dual to $H^1_{\dRB}(A)$. 

We recall that $H_{1,\dRB}(A)$ may be identified with the object $\LiePer E(A)$ of $\CdRB$ defined by the Lie algebra $\Lie E(A)$ of the universal vector extension of $A$ and the subgroup $\Per E(A)_\C$ of $\Lie E(A)_\C$ consisting of the periods of the complex Lie group $E(A)_\C$ (\cf \cite{Bost12}, Section 5.3.3). Moreover, any morphism 
 $$\phi: A \lra B$$
 of abelian varieties over $\Qb$ may be uniquely lifted to a morphism of their universal vector extensions
 $$E(\phi) : E(A) \lra E(B).$$
In turn, $E(\phi)$ defines a morphism $\LiePer E(\phi)$ in $\Hom_{\dRB}(\LiePer E(A), \LiePer E(B))$, which actually coincides with the morphism 
$H_{1, \dRB}(\phi)$ in $\Hom_{\dRB}(H_{1, \dRB}(A), H_{1, \dRB}(B))$ dual to the pullback morphism $H^1_{\dRB}(\phi)$ in
$\Hom_{\dRB}(H^1_{\dRB}(B), H^1_{\dRB}(A)).$ 

In this way, we define functorial maps :
\begin{equation}\label{IsoH1dRB}
\begin{array}{rcccl}
\Hom_{\Qb-\ggp}(A,B) & \lra & \Hom_{\Qb-\ggp}(E(A),E(B)) & \lra & \Hom_{\dRB}(H_{1,\dRB}(A), H_{1,\dRB}(B))   \\
  \phi &  \longmapsto  & E(\phi) & \longmapsto & H_{1,\dRB}(\phi).  
\end{array}
\end{equation}
The first map $\phi \mapsto E(\phi)$ is easily seen to be bijective. Moreover Theorem  \ref{SLmor}, with $G_1= E(A)$ and $G_2 =E(B)$, shows that the second one, which sends $E(\phi)$ to $\LiePer E(\phi) = H_{1,\dRB}(\phi)$, is also bijective (\cf \cite{Bost12}, Theorem 5.3).

Besides, the construction of the de Rham-Betti (co)homology groups is compatible with the duality of abelian varieties. Namely, for any abelian variety $A$ over $\Qb,$ the first Chern class in $H^2_\Gr(A \times A^\wedge, \Z(1))$ of its Poincar\'e line bundle defines an isomorphism in $\CdRB$ (\cf \cite{Bost12}, Section 5.3.3):
\begin{equation}\label{DualdRB}
H_{1,\dRB}(A) \lrasim H^1_{\dRB}(A^\wedge, \Z(1)).
\end{equation}

Let  $A_1$ and $A_2$ be two abelian varieties over $\Qb$. If we compose the isomorphism in Proposition  \ref{prop:DCBiextab}, the fully faithful functor $H_{1,\dRB}$ considered in  (\ref{IsoH1dRB}), and the duality isomorphism (\ref{DualdRB}), we get an isomorphism of $\Z$-modules:
  \begin{multline}\label{preBiextdRB} 
 \Biext^1_{\Qb-\ggp}(A_1,A_2;\Gm) \lra \Hom_{\Qb-\ggp} (A_1, A_2^\wedge) \\ \stackrel{H_{1,\dRB}}{\lra} 
 \Hom_{\dRB}(H_{1,\dRB}(A_1), H_{1,\dRB}(A_2^\wedge)) \\  \lrasim \Hom_{\dRB}(H_{1,\dRB}(A_1), H^1_{\dRB}(A_2, \Z(1))).
 \end{multline}
 Observe that the range of this map may be identified with
 $$\Hom_{\dRB}(H_{1,\dRB}(A_1)\otimes H_{1,\dRB}(A_2), \Z(1))$$
 and also with
 $$\Hom_{\dRB}(\Z(0), H^1_{\dRB}(A_1) \otimes H^1_{\dRB}(A_2, \Z(1))) = :  [H^1_{\dRB}(A_1)\otimes H^1_{\dRB}(A_2)\otimes \Z(1)]_\Gr.$$

We refer the reader to \cite{Deligne74}, Section 10.2, for a discussion of biextension of complex abelian varieties (and more generally, of 1-motives) in the context of Hodge structures, and for diverse equivalent constructions of the map from  $\Biext^1_{\Qb-\ggp}(A_1,A_2;\Gm)$ to 
$$\Hom_{\dRB}(H_{1,\dRB}(A_1)\otimes H_{1,\dRB}(A_2), \Z(1)) \simeq [H^1_{\dRB}(A_1)\otimes H^1_{\dRB}(A_2)\otimes \Z(1)]_\Gr$$
defined by (\ref{preBiextdRB}). We shall content ourselves with the following description of this map. If $L$ denotes the $\Gm$--torsor over $A_1 \times A_2$ defined by some biextension class $\alpha$ of $(A_1,A_2)$ by $\Gm$, its first Chern class in de Rham cohomology $c_{1,\dR}(L)$ defines an element of 
\begin{equation}
\begin{split}
H^2_\dR(A_1 \times A_2/\Qb) & \simeq \wedge^2 H^1_\dR(A_1 \times A_2/\Qb) \simeq  \wedge^2 [H^1_\dR(A_1) \oplus H^1_\dR(A_2/\Qb)] \\
&\simeq \wedge^2 H^1_\dR(A_1/\Qb) \oplus \wedge^2 H^1_\dR(A_2/\Qb) \oplus [H^1_\dR(A_1/\Qb) \otimes H^1_\dR(A_2/\Qb)]
\end{split}
\end{equation}
which actually belongs to the last summand $H^1_\dR(A_1/\Qb) \otimes H^1_\dR(A_2/\Qb).$ 
The map (\ref{preBiextdRB}) sends $\alpha$ to this element
$$B_{A_1,A_2} (\alpha) := c_{1,\dR}(L) \in H^1_\dR(A_1/\Qb) \otimes H^1_\dR(A_2/\Qb).$$

The following theorem summarizes the isomorphisms constructed in the previous paragraphs. They may be seen as counterparts, valid for abelian varieties over $\Qb$ and their de Rham-Betti realizations, of classical facts concerning complex abelian varieties and their Hodge structures (compare for instance the isomorphism   (\ref{BiextdRB}) and \cite{Deligne74}, Construction (10.2.3)).

\begin{theorem}\label{prop:HAB} 1) For any two abelian varieties $A$ and $B$ over $\Qb,$ the map 
 \begin{equation}\label{HAB}
H_{1,\dRB} :  \Hom_{\Qb-\ggp}(A,B) \lrasim \Hom_{\dRB}(H_{1,\dRB}(A), H_{1,\dRB}(B))
\end{equation}
is an isomorphism of $\Z$-modules.

2) For any two abelian varieties $A_1$ and $A_2$ over $\Qb,$ the map 
  \begin{equation}\label{BiextdRB} 
B_{A_1,A_2} :  \Biext^1_{\Qb-\ggp}(A_1,A_2;\Gm) \lrasim [H^1_{\dRB}(A_1)\otimes H^1_{\dRB}(A_2)\otimes \Z(1)]_\Gr
\end{equation}
is an isomorphism of $\Z$-modules.
\end{theorem}

 \subsection{The conjecture $GPC^1$ for abelian varieties and for  products of smooth projective varieties}
 
For any abelian variety $A$ over $\Qb$, the isomorphism 
$$B_{A,A} :  \Biext^1_{\Qb-\ggp}(A,A;\Gm) \lrasim [H^1_{\dRB}(A)\otimes H^1_{\dRB}(A)\otimes \Z(1)]_\Gr$$
maps the subgroup $\SymBiext^1_{\Qb-\ggp}(A,A;\Gm)$ of \emph{symmetric} biextensions in $\Biext^1_{\Qb-\ggp}(A,A;\Gm)$ onto the subgroup 
$[H^1_{\dRB}(A)\otimes H^1_{\dRB}(A)\otimes \Z(1)]^{\rm alt}_\Gr$ of skew-symmetric, or \emph{alternating}, elements in $[H^1_{\dRB}(A)\otimes H^1_{\dRB}(A)\otimes \Z(1)]_\Gr$ (see for instance \cite{Bost12}, 5.3.3 and 5.4 for a discussion of the sign issue involved in this identification).  In turn, $[H^1_{\dRB}(A)\otimes H^1_{\dRB}(A)\otimes \Z(1)]^{\rm alt}_\Gr$ may be identified with $H^2_\Gr(A, \Z(1)),$ and the composite isomorphism
$$NS(A) \stackrel{\Lambda}{\lra} \SymBiext^1_{\Qb-\ggp}(A,A;\Gm) \stackrel{B_{A,A}}{\lra} H^2_\Gr(A, \Z(1))$$
with the first Chern class $c^A_{1,\Gr}$, or equivalently with the classical ``Riemann form".
 
We finally recover the main result of \cite{Bost12}, Section 5: 

\begin{corollary}\label{GPCAb} For any abelian variety $A$ over $\Qb,$
the cycle map establishes an isomorphism of $\Z$-modules:
$$c^A_{1,\Gr}: NS(A) \lrasim  H^2_\Gr(A, \Z(1)).$$

In particular, $GPC^1(A)$ holds.
\end{corollary}

 Finally, we consider two smooth projective varieties  $X$ and $Y$ over $\Qb$, equipped with base points $x\in X(\Qb)$ and $y \in Y(\Qb),$ and their Albanese maps $\alb_{X,x} : X \lra \Alb(X)$ and $\alb_{Y,y}: Y \lra \Alb(Y).$ By pullback, these maps induce isomorphisms in $\CdRB$:
 $$H^1_\dRB(\alb_{X,x}): H^1_\dRB(\Alb(X)) \lrasim H^1_\dRB(X)$$
 and
 $$H^1_\dRB(\alb_{Y,y}): H^1_\dRB(\Alb(Y)) \lrasim H^1_\dRB(Y).$$ 
 
 The K\"unneth decompositions in de Rham and Betti cohomology define an isomorphism in $\CdRB$:
 $$H^2_\dRB (X \times Y)  \lrasim H^2_\dRB (X) \oplus H^2_\dRB (Y) \oplus (H^1_\dRB (X) \otimes H^1_\dRB (Y)),$$
 and consequently an isomorphism of $\Z$-modules:
 $$H^2_\Gr (X \times Y, \Z(1))  \lrasim H^2_\Gr (X,\Z(1)) \oplus H^2_\Gr (Y,\Z(1)) \oplus [H^1_\dRB (X) \otimes H^1_\dRB (Y)\otimes \Z(1)]_\Gr.$$
 
 Moreover the compatibility of the decompositions (\ref{PicXY}) and (\ref{NSXY}) with the K\"unneth decompositions shows that the first Chern class in de Rham-Betti cohomology
 $$c_{1,\Gr}^{X\times Y}: \Pic(X\times Y) \lra H^2_\Gr (X \times Y, \Z(1)),$$
which for any divisor $Z$ in $X\times Y$ maps $[\mathcal O(Z)]$ to $cl^{X\times Y}_\Gr(Z)$, coincides with 
  $c_{1,\Gr}^X$ (resp., with $c_{1,\Gr}^Y$) when restricted to the first (resp., second) summand of the decomposition (\ref{PicXY}) of $\Pic(X\times Y)$, and defines a map
  $$B_{X,Y}: DC(X,Y) \lra [H^1_\dRB (X) \otimes H^1_\dRB (Y)\otimes \Z(1)]_\Gr$$
  by restriction to the third summand. 
  
  The construction of $B_{X,Y}$ is compatible with the Albanese embeddings. Indeed, as one easily checks by unwinding the definitions of the morphisms involved in the above discussion, the following diagram is commutative:
\begin{equation}\label{bigAlb}
 \begin{CD}
\Biext^1_{\Qb-\ggp}(\Alb(X), \Alb(Y);\Gm) @>{\sim}>> DC(X,Y) \\
@V{B_{\Alb(X), \Alb(Y)}}VV                @VV{B_{X,Y}}V \\
[H^1_\dRB (\Alb(X)) \otimes H^1_\dRB (\Alb(Y))\otimes \Z(1)]_\Gr @>{\sim}>>       [H^1_\dRB (X) \otimes H^1_\dRB (Y)\otimes \Z(1)]_\Gr,
\end{CD}
\end{equation}      
  where the upper (resp., lower) horizontal arrow is the isomorphisms deduced from Propositions \ref{prop:DCAlb} and \ref{prop:DCBiextab} (resp., the isomorphism $H^1_\dRB(\alb_{X,x}) \otimes H^1_\dRB(\alb_{Y,y})\otimes \rm{Id}_{\Z(1)}$).
  
 According to Theorem \ref{prop:HAB}, 2),  the left vertical arrow $B_{\Alb(X), \Alb(Y)}$ in (\ref{bigAlb}) is an isomorphism. Together with the previous discussion, this establishes the following:
  
\begin{corollary}\label{GPCproduct} For any two smooth projective varieties $X$ and $Y$ over $\Qb,$
the map 
\begin{equation}\label{DCproduct}
B_{X,Y}: DC(X,Y) \lra [H^1_{\dRB}(X)\otimes H^1_{\dRB}(Y)\otimes \Z(1)]_\Gr
\end{equation}
is an isomorphism of $\Z$-modules,
and consequently
\begin{equation}\label{product}
GPC^1(X) \mbox{ and } GPC^1(Y) \Longleftrightarrow GPC^1(X\times Y).
\end{equation}
\end{corollary}

As observed above (see \ref{statement}), for trivial reasons, $GPC^1(X)$ holds for any smooth projective curve $X$ over $\Qb$. Consequently, Corollary \ref{GPCproduct} implies the validity of $GPC^1(X)$ for any product $X$ of smooth projective curves over $\Qb$.

\section{Weights in degree 1 and the second cohomology groups of smooth open varieties}

In this section, we study the generalization of the Grothendieck period conjecture $GPC^k$ to smooth non-proper varieties over $\Qb$ , mainly when $k=1$, and we establish the birational invariance of $GPC^1$. 

Specifically, let $X$ be a smooth quasi-projective variety over $\Qb$. According to \cite{Grothendieck66}, we may still consider the algebraic de Rham  cohomology groups of $X$ over $\Qb$,
$$H^i_{\dR}(X/\Qb) := \H^i(X, \Omega^\bullet_{X/\Qb}),$$
and the comparison isomorphisms
(\ref{compdRC}) and (\ref{compBC}) still hold in this quasi-projective setting. Moreover the definitions of cycles classes in de Rham and Betti cohomology also extend, and Proposition  \ref{compdRB} is still valid. 

As a consequence, we may introduce the de Rham-Betti cohomology groups of $X$, $H^k_{\dRB}(X, \Z(j))$ -- as before, defined in terms of the Betti cohomology modulo torsion -- and $H^k_{\dRB}(X, \Q(j))$, as well as the $\Q$-vector spaces $$H^{2k}_{\Gr}(X, \Q(k)) := H^{2k}_{\dR}(X/\Qb) \cap H^{2k}(X^\an_{\C}, \Q(k))$$ and the cycle 
 map
$$\cl^X_{\Gr}: Z^k(X)_{\Q} \lra H^{2k}_{\Gr}(X, \Q(k)).$$
We shall say that $GPC^k(X)$ holds when this map is onto.

Here again, our main technical tool will be a transcendence theorem \emph{\`a la} Schneider-Lang, which will allow us to establish a purity property of classes in $H^2_{\Gr}(X, \Q(1)).$ This result and its proof suggest some conjectural weight properties of the cohomology classes
in $$H^k_\Gr(X,\Q(j)) :=
 H^{k}_{\dR}(X/\Qb) \cap H^{k}(X^\an_{\C}, \Q(j))
$$ 
that we discuss at the end of this section.

\subsection{Transcendence and $H^1$}\label{transH1section}

\begin{theorem}\label{transH1} For any smooth quasi-projective variety $X$ over $\Qb$, we have, inside
 $H^1_{\dR}(X_{\C}/\C) \simeq H^1(X(\C), \C)$: 
$$H^1_{\dR}(X/\Qb) \cap H^1(X_\C^\an, \Q) =\{0\}.$$
In other words:
$$H^1_{\Gr}(X, \Q(0))=\{0\}.$$
\end{theorem}

When $X$ is $\Gm$, this theorem precisely asserts the transcendence of $\pi,$ and is equivalent to Corollary \ref{corSLmorGa} for $G=\Gm.$ 

Actually a considerable strengthening of Theorem \ref{transH1} is known to hold (\cf \cite{BakerWuestholz07}, notably Corollary 6.9): 
\emph{for any cohomology class $\alpha$ in  $H^1_{\dR}(X/\Qb)$} (identified to a subspace of  $H^1(X^{\an}_{\C}, \C)$) \emph{and any $\gamma \in H_{1}(X(\C), \Z)$, the integral $\int_{\gamma} \alpha$ 
 either vanishes, or belongs to $\C\setminus\Qb.$} 
 This follows from the so-called ``analytic subgroup theorem"   of W\"ustholz (\cite{Wuestholz89}) --- a generalized version of Baker's transcendence results on linear forms in logarithms, valid over any commutative algebraic group over $\Qb$ --- combined with the construction of generalized Albanese varieties in \cite{FaltingsWuestholz84}. 
 
 For the sake of completeness, we sketch a proof of Theorem \ref{transH1} based on the less advanced transcendence results,  \emph{\`a la} Schneider-Lang, recalled in Section \ref{TransSL}. 
 
 \begin{proof}[Proof of Theorem \ref{transH1}] 1) Assume first that $X$ is projective. Then 
 $H^1_{\dR}(X/\Qb)$ may be identified with the Lie algebra of the universal vector extension $$E_{X/\Qb} := E(\Pic^0_{X/\Qb})$$ of the Picard variety $\Pic^0_{X/\Qb}$, which classifies algebraically trivial line bundles over $X$. Moreover the canonical isomorphism 
 $$\Lie E_{X/\Qb} \lrasim H^1_{\dR}(X/\Qb)$$
 defines, after extending the scalars from $\Qb$ to $\C$ and composing with the comparison isomorphism (\ref{compdRC}), an isomorphism of complex vector spaces
 $$\Lie E_{X/\Qb, \C} \lrasim  H^1_{\dR}(X/\Qb) \otimes_\Qb \C \lrasim H^1(X_\C^\an, \C)$$
 which maps $\Per E_{X/\Qb, \C}$ to the subgroup $$H^1(X_\C^\an, \Z(1)) = 2 \pi i H^1(X_\C^\an, \Z)$$ of $H^1(X_\C^\an, \C)$
 (see for instance \cite{Messing73}, \cite{MazurMessing74}, and \cite{BK09}, Appendix B).
 
 Therefore, applied to $G = E_{X/\Qb},$ Corollary \ref{SLmorGm} shows that
 \begin{equation}\label{Eint}
 \Hom_{\Qb-\ggp}(\Gm_{,\Qb}, E_{X/\Qb}) \lrasim H^1_{\dR}(X/\Qb) \cap H^1(X_\C^\an, \Z),
\end{equation}
where the intersection is taken in $H^1_{\dR}(X_\C/\C) \simeq H^1(X_\C^\an, \C).$ 

Now the algebraic group $E_{X/\Qb}$ is an extension of an abelian variety by a vector group, and there is no non-zero morphism of algebraic groups from $\Gm_{,\Qb}$ to $E$. Finally (\ref{Eint}) shows that
$$H^1_{\dR}(X/\Qb) \cap H^1(X_\C^\an, \Z) = \{0\},$$ or equivalently
$$H^1_{\dR}(X/\Qb) \cap H^1(X_\C^\an, \Q) = \{0\}.$$

2) In general, we may consider a smooth projective variety $\oli{X}$ over $\Qb$ containing $X$ as an open dense subvariety. Let $(Y_i)_{i\in I}$ be the irreducible components of $\oli{X} \setminus X$ of codimension 1 in $\oli{X}$. The inclusion morphism $i : X \hlra \oli{X}$ and the residue maps along the components $Y_i$ of $\oli{X} \setminus X$ determine a commutative diagram with exact lines (compare with the diagram (\ref{BigGysin}) in the proof of Proposition  \ref{HGrXU}, \emph{infra}):
\begin{equation}\label{1Gysin}
\begin{CD}
0@>>>    H^1_\dR(\oli{X}/\Qb) @>{i^\ast_\dR}>> H^1_\dR(X/\Qb) @>{\Res_\dR}>> \Qb^I \\
@.        @VVV          @VVV                            @VVV   \\
0@>>>    H^1(\oli{X}^\an_\C, \C) @>{i^\ast_\C}>> H^1(X^\an_\C, \C) @>{\Res_\C}>> \C^I  \\
@.  @AAA @AAA  @AAA  \\
0@>>>    H^1(\oli{X}^\an_\C, \Q(1)) @>{i^\ast_\rB}>> H^1(X^\an_\C, \Q(1)) @>{\Res_\rB}>> \Q^I.  \\
\end{CD}
\end{equation}
The vertical arrows in this diagram are injections (defined, in the first two columns, by the comparison isomorphisms (\ref{compdRC}) for $X$ and $\oli{X}$, and the inclusion of $\Q(1)$ into $\C$) that will be written as inclusions in the sequel, and the middle line may be identified with the tensor product with $\C$ over $\Qb$ (resp. over $\Q$) of the first (resp. third) one.

We need to show that any element in the intersection of $H^1_\dR(X/\Qb)$ and $H^1(X^\an_\C, \Q)$ actually vanishes. Let $\alpha$ be such an element in $H^1_\dR(X/\Qb) \cap H^1(X^\an_\C, \Q)$.  Its residue $\Res_\C \alpha$ belongs to $\Qb^I$ (since it is also $\Res_\dR \alpha$) and to $(2\pi i)^{-1} \Q^I$  (since it may also be written  $(2\pi i)^{-1} \Res_\rB(2\pi i\alpha)$). The transcendence of $2\pi i$ now shows that $\Res_\C \alpha$ vanishes, and the exactness of the lines in  (\ref{1Gysin}) that $\alpha$ belongs to (the image by $i_\C^\ast$ of) $H^1_\dR(\oli{X}/\Qb) \cap H^1(\oli{X}^\an_\C, \Q)$. According to the first part of the proof, this intersection vanishes.
\end{proof}

Observe that Part 1) of the proof of Theorem \ref{transH1}, with $\Gm$ replaced  by $\Ga$ and Corollary \ref{SLmorGm} by Corollary \ref{corSLmorGa}, establishes the following:

\begin{theorem}\label{transH1bis} For any smooth projective variety $X$ over $\Qb$, we have, inside
 $H^1_{\dR}(X_{\C}/\C) \simeq H^1(X^\an_\C, \C)$ : 
$$H^1_{\dR}(X/\Qb) \cap  H^1(X^\an_\C, \Q(1)) =\{0\}.$$
In other words, we have:
$$H^1_{\Gr}(X, \Q(1))=\{0\}.$$
\end{theorem}

\subsection{Purity of  $H^2_{\Gr}(U, \Q(1))$}\label{purH2}

\begin{proposition}\label{HGrXU} Let $X$ be a smooth projective variety over $\Qb$ and $U$ a dense open subscheme of $X$. Let $i: U \hookrightarrow X$ denote the inclusion morphism,  $(D_{\alpha})_{1\leq \alpha \leq A}$ the irreducible components of $X \setminus U$ of codimension $1$ in $X$, and $([D_{\alpha}])_{1\leq \alpha \leq A}:= (\cl^X_{\Gr}(D_{\alpha}))_{1\leq \alpha \leq A}$ their images in $H^2_{\Gr}(X,\Q(1)).$  

Then the following diagram of $\Q$-vector spaces
\begin{equation}\label{diagZH}
\begin{CD}
0 @>>> \Q^A      @>{(D_{1},\ldots,D_{A})}>>     Z^1(X)_{\Q} @>{i^\ast}>> Z^1(U)_{\Q} @>>> 0 \\
@.        @|          @VV{\cl^X_{\Gr}}V                            @VV{\cl^U_{\Gr}}V   @. \\
@. \Q^A @>{([D_{1}],\ldots,[D_{A}])}>> H^2_{\Gr}(X,\Q(1)) 
@>{i^\ast_{\Gr}}>> H^2_{\Gr}(U,\Q(1)) 
@>>> 0 \\
\end{CD}
\end{equation}
is commutative with exact lines.
\end{proposition}

This directly implies:
\begin{corollary}\label{corUX} The $\Q$-linear map $\cl^X_{\Gr}: Z^1(X)_{\Q} \lra H^2_{\Gr}(X,\Q(1))$ is onto iff $\cl^U_{\Gr}:Z^1(U)_{\Q} \lra H^2_{\Gr}(U,\Q(1))$ is onto.
\noindent In other words,
$$GPC^1(X) \Longleftrightarrow GPC^1(U).$$
\end{corollary}

Let us emphasize that the ``non-formal'' part of the proof of Proposition \ref{HGrXU} is the surjectivity of the map 
$$i^\ast: H^2_{\Gr}(X, \Q(1)) \lra H^2_{\Gr}(U, \Q(1)).$$
It shows (and is basically equivalent to the fact) that  $H^2_{\Gr}(U, \Q(1))$ is included in the weight zero part $W_0H^2(U^\an_\C, \Q(1))$ of $H^2(U^\an_\C, \Q(1))$. This purity result will be deduced from the transcendence properties of the $H^1$ recalled in Theorem \ref{transH1}, applied to  components  of codimension $1$ of $X\setminus U$ . 

Corollary  \ref{corUX}  implies the birational invariance of $GPC^1(X)$.
From this result, together with the compatibility of direct images of cycles with trace maps in de Rham and Betti cohomology, one easily derives that, more generally, 
for any two smooth projective varieties $X$ and $Y$ over $\Qb,$
if there exists a dominant rational map $f: X \dasharrow Y,$ then 
$GPC^1(X)$ implies
$GPC^1(Y).$ (Compare \cite{Tate94}, (5.2) Theorem (b).) This is also a special case of our results in Section 5 (\cf Corollary \ref{dominant}).

\begin{proof}[Proof of Proposition \ref{HGrXU}]
 The commutativity of (\ref{diagZH}) and the exactness of its first line are clear. We are left to establish the exactness of its second line. 
 
 Let us consider $F:= X \setminus U$, the union $F^{>1}$ of its irreducible component of codimension strictly bigger than $1$, and the closed subset $F_{\rm sing}$ of non-regular points of $F$. Observe that, since $F_{\rm sing} \cup F^{>1}$ has codimension strictly bigger than $1$ in  $X$, the inclusion $j: V \hlra X$ of the open subscheme $$V := X \setminus (F_{\rm sing} \cup F^{>1})$$ induces compatible isomorphisms between de Rham and Betti cohomology groups, for $i\in\{0, 1, 2\}$:
 \begin{equation}\label{jiso}
 \begin{CD}
H^{i}_\dR(X/\Qb) @>{\sim}>> H^i_{\dR} (V/\Qb)\\
@VVV        @VVV                                     \\
H^{i}(X_\C^\an, \C) @>{\sim}>> H^i(V^\an_\C, \C)  \\
@AAA @AAA  \\
H^{i}(X_\C^\an, \Q(1)) @>{\sim}>> H^i (V^\an_\C, \Q(1)). \\
\end{CD}
 \end{equation}
 
 The open subscheme $U:= X \setminus F$ is contained in $V$. Moreover
 $$D:= V \setminus U = F \setminus (F_{\rm sing} \cup F^{>1}) = \bigcup_{1 \leq \alpha \leq A} D_{\alpha} \setminus (F_{\rm sing} \cup F^{>1})$$
 is a closed smooth divisor in $V$, with irreducible components
 $${\overset{\circ}{D}}_{\alpha} := D_{\alpha}\setminus (F_{\rm sing} \cup F^{>1}), \qquad 1 \leq \alpha \leq A.$$
 The inclusions $D \hlra V$ and $D_{\C}\hlra V_{\C}$ define compatible Gysin isomorphisms with value in the cohomology with support:
\begin{equation}\label{smallGysin}\begin{CD}
H^{i-2}_\dR(D/\Qb) @>{\sim}>> H^i_{\dR,D} (V/\Qb)\\
@VVV        @VVV                                     \\
H^{i-2}(D_\C^\an,\C) @>{\sim}>> H^i_{D^\an_\C} (V^\an_\C, \C)  \\
@AAA @AAA  \\
H^{i-2}(D_\C^\an,\Q) @>{\sim}>> H^i_{D^\an_\C} (V^\an_\C, \Q(1)). \\
\end{CD}
\end{equation}
Therefore the long exact sequences of cohomology groups, relating the cohomology of $V$ with support in $D$, the cohomology of $V$ and the cohomology of $U =V \setminus D$, may be interpreted as a ``Gysin exact sequences'' which, combined with the isomorphisms (\ref{jiso}),   fits into a commutative diagram with exact lines: 

\begin{equation}\label{BigGysin}
\begin{CD}
H^0_\dR(D/\Qb) @= \Qb^A      @>{\gamma_\dR}>>    H^2_\dR(X/\Qb) @>{i^\ast_\dR}>> H^2_\dR(U/\Qb) @>{\Res_\dR}>> H^1_\dR(D/\Qb) \\
@.        @VVV          @VVV                            @VVV   @VVV\\
H^0(D^\an_\C, \C) @= \C^A      @>{\gamma_\C}>>    H^2(X^\an_\C, \C) @>{i^\ast_\C}>> H^2(U^\an_\C, \C) @>{\Res_\C}>> H^1(D^\an_\C, \C)  \\
@.  @AAA @AAA  @AAA @AAA \\
H^0(D^\an_\C, \Q) @= \Q^A      @>{\gamma_\rB}>>    H^2(X^\an_\C, \Q(1)) @>{i^\ast_\rB}>> H^2(U^\an_\C, \Q(1)) @>{\Res_\rB}>> H^1(D^\an_\C, \Q).  \\
\end{CD}
\end{equation}

In (\ref{jiso}), (\ref{smallGysin}) and (\ref{BigGysin}), the vertical arrows are injections, that we shall write as inclusions in the sequel. The middle line may be identified with the tensor product with $\C$ over $\Qb$ (resp. over $\Q$) of the first (resp. third) one. By definition, the map $\gamma_{\C}$ (resp. $\gamma_{\dR}$, $\gamma_{\rB}$) maps any $A$-tuple $(\lambda_{\alpha})_{1\leq \alpha \leq A}$ in $\C^A$ (resp. in $\Qb^A$, $\Q^A$) to $\sum_{1\leq \alpha \leq A} \lambda_{\alpha} [D_{\alpha}]$.

Recall also that, for $\Q$-divisors on the smooth projective variety $X$, homological and numerical equivalence coincide (see for instance \cite{FultonIT}, 19.3), and that, if $d:= \dim X,$ we have compatible isomorphisms of one-dimensional vector spaces: 
$$\begin{CD}
H^{2d}(X/\Qb) @>{\sim}>{\Tr_{\dR}}> \Qb \\
@VVV        @VVV \\
H^{2d}(X^\an_\C, \C) @>{\sim}>{\Tr_\C}> \C\\ 
@AAA     @AAA\\ 
H^{2d}(X^\an_\C, \Q(d)) @>{\sim}>{\Tr_{\rB}}> \Q. \\
\end{CD}
$$

  Consequently, if $D$ denotes the dimension of the $\Q$-vector space $\im \gamma_{\rB}$, we may choose a $B$-tuple $(C_1,\ldots, C_D)$ of elements of $Z_1(X)$ such that the map 
  $$
\begin{array}{rrcl}
\psi :&H^{2}(X^\an_\C, \C)  & \lra  & \C^D   \\
 & c & \longmapsto   &(\Tr_\C (c.[C_i]))_{1\leq i \leq D}   
\end{array}
$$
--- where $[C_i]$ denotes the cycle class of $C_i$ in $H^{2d-2}_{\Gr}(X, \Q(d-1))$ --- defines by restriction isomorphisms of $\Q$-, $\Qb$-, and $\C$-vector spaces:
$$\psi_{\rB} : \mathrm{Im} \gamma_{\rB} \lrasim \Q^D,$$
$$\psi_{\dR} : \mathrm{Im} \gamma_{\dR} \lrasim \Qb^D,$$
and $$\psi_{\C} : \mathrm{Im} \gamma_{\C} \lrasim \C^D.$$

Consider a class $c$ in $H^2_{\Gr}(U, \Q(1))$. Its image under $\Res_\C$ belongs to $\mathrm{Im} \Res_{\dR} \cap \mathrm{Im} \Res_{\rB},$ hence to the subspace 
$$H^1_{\dR}(D/\Qb) \cap H^1(D^\an_\C,\Q) = \bigoplus_{1\leq \alpha \leq A} H^1_{\dR}(\overset{\circ}{D}_\alpha/\Qb) \cap H^1(\overset{\circ}{D}^\an_{\alpha,\C},\Q)$$
of 
$$H^1(D^\an_\C,\C) = \bigoplus_{1\leq \alpha \leq A} H^1(\overset{\circ}{D}^{\an}_{\alpha,\C},\C).$$

According to Theorem \ref{transH1}, this intersection vanishes, and therefore we may find $\alpha \in H^2_\dR (X/\Qb)$ and $\beta \in H^2(X^\an_\C, \Q(1))$ such that
$$c = i^\ast_{\dR}\alpha = i^\ast_{\rB}\beta.$$

The class $\delta := \beta - \alpha$ in  $H^2(X^\an_\C, \C)$ satisfies 
$i_\C^\ast \delta = 0,$ hence belongs to $\mathrm{Im} \gamma_\C$. Moreover 
$\psi_\C (\delta) = \psi_{\rB}(\beta) - \psi_{\dR}(\alpha)$ belongs to $\Qb^D$. Consequently $\delta$ belongs to $\mathrm{Im} \gamma_{\dR}$, and finally $\beta = \alpha + \delta$ belongs to $H^2_\Gr (X, \Q(1))$ and is mapped to $c$ by $i^\ast_\Gr.$

This establishes the surjectivity of $i^\ast_\Gr$ in the second line of (\ref{diagZH}). Its exactness then follows 
 from the exactness at $H^2(X^\an_\C, \Q(1))$ of the last line of (\ref{BigGysin}).
\end{proof}

\subsection{Periods and the weight filtration}\label{wfiltration}

The reader will have noticed that the arguments of the preceding sections essentially reduce to reasoning on weights. We briefly discuss what relationship one might expect between the weight filtration and Grothendieck classes.

Let $X$ be a smooth quasi-projective variety over $\Qb$. As shown in \cite{Deligne71} -- see also \cite{Deligne75} -- both the algebraic de Rham cohomology groups of $X$ and its Betti cohomology groups with rational coefficients are endowed with a canonical \emph{weight filtration} $W_\bullet$. This is an increasing filtration on cohomology. The group $W_n H^k(X^\an_\C, \Q(j))$ is the subspace of weight at most $n$ in $H^k(X^\an_\C, \Q(j))$, and the group $\mathrm{Gr}^n_{W_\bullet} H^k(X^\an_\C, \Q(j))$ is the ``part" of weight $n$. The weight filtration is functorial and compatible with products. 

If the smooth variety $X$ is projective, the group $H^k(X^\an_\C, \Q(j))$ is of pure weight $k-2j$, meaning that $\mathrm{Gr}^n_{W_\bullet} H^k(X^\an_\C, \Q(j))$ vanishes unless $n=k-2j$. In general, the weights of $H^k(X^\an_\C, \Q(j))$ all lie between $k-2j$ and $2k-2j$ as proved in \cite{Deligne71}, meaning that $W_{k-2j-1}H^k(X^\an_\C, \Q(j))=0$ and $W_{2k-2j}H^k(X^\an_\C, \Q(j))=H^k(X^\an_\C, \Q(j))$.

The results above hold with Betti cohomology replaced by de Rham cohomology, and the weight filtration is compatible with the comparison isomorphism between de Rham and Betti cohomology ; see \cite{Jannsen90}, chapter 3. As a consequence, the de Rham-Betti cohomology groups of $X$ are endowed with a weight filtration $W_{\bullet}$ as well, that is sent to the weight filtration above on both the de Rham and the Betti realization. The graded objects $\mathrm{Gr}^n_{W_\bullet} H^k_\dRB(X, \Q(j))$ are also objects of the category $\CdRBQ$.

\bigskip

Along the lines of \ref{conjectures}, it might be sensible to formulate the following conjecture.

\begin{conjecture}\label{weights}
Grothendieck classes on smooth quasi-projective varieties live in weight zero. 

In other words, let $X$ be a smooth quasi-projective variety over $\Qb$, and let $j, k$ be two integers. Then 
\begin{enumerate}
\item $W_{-1} H^k_\dRB(X, \Q(j))_\Gr=0$.
\item The natural injection 
$$W_{0} H^k_\dRB(X, \Q(j))_\Gr\hookrightarrow H^k_\dRB(X, \Q(j))_\Gr$$
is an isomorphism.
\end{enumerate}
\end{conjecture}

The results of sections \ref{transH1section} and \ref{purH2} may be rephrased as partial results towards Conjecture \ref{weights}.

\begin{theorem}
Conjecture \ref{weights} holds if $X$ is a smooth quasi-projective variety and $(j, k)$ is equal to 
$(0, 1), (1,1)$ or $(1,2)$.
\end{theorem}

\begin{proof}
The statement is exactly what is proved in Theorem \ref{transH1}, Theorem \ref{transH1bis} and Proposition \ref{HGrXU} respectively.
\end{proof}

\section{Absolute Hodge classes and the Grothendieck period conjecture}

In this section, we explain how some well-known results regarding absolute Hodge cycles and the conjectures of Hodge and Tate may be transposed into the setting of the Grothendieck period conjecture.

\subsection{Absolute Hodge classes}

The natural cohomological setting that relates the Hodge conjecture and the Grothendieck period conjecture is the one of absolute Hodge classes, as introduced by Deligne in \cite{DeligneMilneOgusShih}. While it is not strictly necessary to introduce absolute Hodge classes to prove the results in this section, as one can rely on André's motivated classes only, consider it  worthwhile to compare the definition of Grothendieck classes to that of absolute Hodge classes. 
We refer to \cite{DeligneMilneOgusShih} and the survey \cite{CharlesSchnell11} for details on absolute Hodge classes.

\bigskip

Let $X$ be a smooth projective variety over an algebraically closed field $K$ of finite transcendence degree over $\Q$. If $\sigma$ is an embedding of $K$ into $\C$, let $\sigma X$ be the complex variety deduced from $X$ by  the base field extension $\sigma:K \ra \C$.

\begin{definition}
Let $\alpha$ be a cohomology class in $\HdR^{2k}(X/K)$.
\begin{enumerate}
\item Let $\sigma$ be an embedding of $K$ into $\C$. We say that $\alpha$ is \emph{rational relative to $\sigma$} if the image of $\alpha$ in $\HdR^{2k}(\sigma X/\C)$ belongs to the image of the Betti cohomology group $H^{2k}(\sigma X^\an, \Q(k))$ under the comparison isomorphism (\ref{compdRC}).
\item The class $\alpha$ is a \emph{Hodge class relative to $\sigma$} if it is rational relative to $\sigma$ and $\alpha$ belongs to $F^k\HdR^{2k}(X/K)$, where $F^{\bullet}$ is the Hodge filtration.
\item The class $\alpha$ is an \emph{absolute rational class} if $\alpha$ is rational relative to all embeddings of $K$ into $\C$.
\item The class $\alpha$ is an \emph{absolute Hodge class} if it is an absolute rational class and belongs to $F^k\HdR^{2k}(X/K)$.
\item Given an embedding $\sigma$ of $K$ into $\C$, we say that a class $\beta$ in $H^{2k}(\sigma X^\an, \Q(k))$ is \emph{absolute rational} (resp. \emph{absolute Hodge}) if its image under the comparison isomorphism (\ref{compdRC}) is absolute rational (resp. absolute Hodge).
\end{enumerate}
\end{definition}

Observe that, in the preceding definition, when $K$ is the field $\Qb$ and $\sigma$ is the inclusion of $\Qb$ in $\C$, the set of classes in $\HdR^{2k}(X/K)$ rational relative to $\sigma$ is the group $H^{2k}_{\Gr}(X, \Q(k))$. 

\begin{proposition}\label{propAlgHodgeGr}
Let $X$ be a smooth projective variety over $\Qb$, $k$ a non-negative integer, $K$ an algebraically closed field of finite transcendence degree over $\Q$ containing $\Qb$. Let $\alpha$ be a class in $\HdR^{2k}(X_K/K)$. Then each of the following conditions imply the following.
\begin{enumerate}
\item The class $\alpha$ is algebraic.
\item The class $\alpha$ is motivated \emph{(in the sense of Andr\'e \cite{Andre96b}).}
\item The class $\alpha$ is an absolute Hodge class.
\item The class $\alpha$ is an absolute rational class.
\item The class $\alpha$ lies in $H^{2k}_{\Gr}(X, \Q(k))$.
\end{enumerate}
\end{proposition}

\begin{proof}
The only step that does not formally follow from the definitions is the fact that absolute rational classes lie in $H^{2k}_{\Gr}(X, \Q(k))$. By the observation above, this reduces to proving that if $\alpha$ is an absolute rational class in $\HdR^{2k}(X_K/K)$, then $\alpha$ is defined over $\Qb$, that is, $\alpha$ belongs to the subspace $\HdR^{2k}(X/\Qb)$. This is proven in \cite{DeligneMilneOgusShih}, Corollary 2.7.
\end{proof}

The question whether (4) implies (3) is asked by Deligne in \cite{DeligneMilneOgusShih}, Question 2.4. 

\bigskip

Let $X$ and $Y$ be two smooth projective algebraic varieties over $\Qb$.  As explained in \cite{DeligneMilneOgusShih} (see also \cite{CharlesSchnell11}, 11.2.6), the definition of an absolute Hodge class given above can be extended to that of an \emph{absolute Hodge} morphism
$$f:  H^{2k}(X^\an_\C, \Q(k)) \lra H^{2l}(Y^\an_\C, \Q(l)).$$ 
Note that, as a consequence of Proposition \ref{propAlgHodgeGr}, such a morphism maps $\HdR^{2k}(X/\Qb)$ to $\HdR^{2l}(Y/\Qb)$ and $H^{2k}_{\Gr}(X, \Q(k))$ to $H^{2l}_{\Gr}(Y, \Q(l))$.

The following proposition somehow asserts the motivic nature of the Grothendieck period conjecture. Its proof relies heavily on the existence of polarizations.

\begin{proposition}\label{motivic}
Let $X$ and $Y$ be two smooth projective varieties over $\Qb,$ and let $k$ and $l$ be two non-negative integers.
Let
$$f: H^{2k}(X^\an_\C, \Q(k)) \lra H^{2l}(Y^\an_\C, \Q(l))$$ 
be an absolute Hodge morphism.

\begin{enumerate}
\item 
We have 
$$f(H^{2k}_{\Gr}(X, \Q(k)))=H^{2l}_{\Gr}(Y, \Q(l))\cap \mathrm{Im}(f).$$
\item Assume that $X$ satisfies the Hodge conjecture in codimension $k$ and that $Y$ satisfies the Grothen\-dieck period conjecture in codimension $l$. Then we have 
$$f(H^{2k}_{\Gr}(X, \Q(k)))=H^{2l}_{\Gr}(Y, \Q(l))\cap \mathrm{Im}(f)=f(cl^X_\Gr(Z^k(X)_{\Q})).$$
In particular, if $f$ is injective, $X$ satisfies the Grothendieck period conjecture in codimension $k$.
\item Assume that $f$ is algebraic, \emph{namely, that $f$ is induced by an algebraic correspondence between $X$ and $Y$,} and that $X$ satisfies the Grothendieck period conjecture in codimension $k$. Then 
$$H^{2l}_{\Gr}(Y, \Q(l))\cap \mathrm{Im}(f)=cl^Y_\Gr(Z^k(Y)_{\Q})\cap \mathrm{Im}(f).$$
In particular, if $f$ is surjective, then $Y$ satisfies the Grothen\-dieck period conjecture in codimension $l$.
\end{enumerate}
\end{proposition}

\begin{proof}
\begin{enumerate}
\item This is a semisimplicity result that relies in an essential way on polarizations. By \cite{CharlesSchnell11}, Proposition 24 and Corollary 25, there exists an absolute Hodge morphism
$$g : H^{2l}(Y^\an_\C, \Q(l))\lra H^{2k}(X^\an_\C, \Q(k))$$
such that the restriction of $g$ to the image of $f$ is a section of $f$. Now if $\beta$ is an element of $H^{2l}_{\Gr}(Y, \Q(l))\cap \mathrm{Im}(f)$, $\alpha:=g(\beta)$ belongs to $H^{2k}_{\Gr}(X, \Q(k))$ and $f(\alpha)=\beta$.

\item Let $g$ be as above, and let $\alpha$ be an element of $H^{2k}_{\Gr}(X, \Q(k))$. Then $f(\alpha)$ lies in $H^{2l}_{\Gr}(Y, \Q(l))\cap \mathrm{Im}(f)$. Since $Y$ satisfies the Grothendieck period conjecture in codimension $l$, $f(\alpha)$ is the cohomology class of an algebraic cycle on $Y$. In particular, $f(\alpha)$ is a Hodge class. 

Since $g$ is absolute Hodge, $\alpha'=g(f(\alpha))$ is a Hodge class as well, hence the class of an algebraic cycle on $X$ since $X$ satisfies the Hodge conjecture in codimension $k$. By the definition of $g$, $\alpha$ and $\alpha'$ have the same image by $f$.

\item Let $g$ be as above, and let $\beta$ be an element of $H^{2l}_{\Gr}(Y, \Q(l))\cap \mathrm{Im}(f)$. Then $\alpha=g(\beta)$ belongs to $H^{2k}_{\Gr}(X, \Q(k))$. Since $X$ satisfies the Grothendieck period conjecture in codimension $k$, $\alpha$ is the cohomology class of an algebraic cycle. Since $f$ is algebraic, $\beta=f(\alpha)$ is the cohomology class of an algebraic cycle as well.
\end{enumerate}
\end{proof}

\begin{corollary}\label{dominant}
Let  $X$ and $Y$ be two smooth projective varieties over $\Qb.$

If there exists a dominant rational map $f: X \dasharrow Y,$ then 
\begin{equation}\label{XY}
GPC^1(X) \Longrightarrow GPC^1(Y).
\end{equation}
\end{corollary}

\begin{proof}
Let $\Gamma$ be the graph of $f$ in $X\times Y$, and let $\pi : \Gamma'\ra \Gamma$ be a resolution of singularities of $\Gamma$. Since $\Gamma'$ is birational to $X$, Corollary \ref{corUX} shows that $GPC^1(X)$ is equivalent to $GPC^1(\Gamma')$. The second projection from $\Gamma'$ to $Y$ is dominant. As a consequence, up to replacing $X$ with $\Gamma'$, we can assume that $f$ is a \emph{morphism} from $X$ to $Y$.

Consider the map 
$$f^* : H^2(Y^\an_\C, \Q(1))\lra H^2(X^\an_\C, \Q(1)).$$
It is an absolute Hodge morphism, and is well-known to be injective. Indeed, if $V$ is a subvariety of $X$ such that the dimension of $V$ is equal to the dimension of $Y$ and the restriction of $f$ to $V$ is dominant, and if $[V]$ is the cohomology class of $V$, then the map 
$$H^2(Y^\an, \Q(1))\lra H^2(Y^\an, \Q(1)), \alpha\longmapsto f_*(f^*\alpha\cup[V])$$
is equal to the multiplication by the degree $[\Qb(V):\Qb(Y)]$ of $V$ over $Y$. 

By Proposition \ref{motivic}, (2),  and Lefschetz's theorem on $(1,1)$-classes, this proves that $GPC^1(X)$ implies $GPC^1(Y)$.
\end{proof}

\begin{corollary}\label{GPC1hyp}
Let $X$ be a smooth projective variety over $\Qb$ of dimension at least $3$, and let $Y$ be a smooth hyperplane section of $X$ defined over $\Qb$. Then 
$$GPC^1(Y)\implies GPC^1(X).$$
\end{corollary}

\begin{proof} It is again a consequence of Proposition \ref{motivic}, (2),  and Lefschetz's theorem on $(1,1)$-classes, applied to the (algebraic, hence absolute Hodge) morphism
$$i^*:  H^{2}(X^\an_\C, \Q(k))\lra H^{2}(Y^\an_\C, \Q(k))$$
defined by the inclusion map $i : Y \hra X.$ Indeed, according to the weak Lefschetz theorem, $i^*$ is injective. 
\end{proof}

Observe that, as pointed out in Section 1.2, Corollary \ref{GPC1hyp} shows that the validity of $GPC^1(X)$ for arbitrary smooth projective varieties would follow from its validity for smooth projective surfaces. 

Observe also that, when the dimension of $X$ is greater than 3, Corollary \ref{GPC1hyp} is a straightforward consequence of the classical weak Lefschetz theorems for cohomology and Picard groups, which show that when this dimension condition holds, the injection $i: Y\hlra X$ induces isomorphisms
$$i^\ast : H^2_{\Gr} (X, \Q(1)) \lrasim H^2_{\Gr}(Y, \Q(1))$$
and $$i^\ast: \Pic (X) \lrasim \Pic (Y).$$
Accordingly, the actual content of Corollary \ref{GPC1hyp} concerns the case where $X$ is a threefold and $Y$ is a surface.

\subsection{Abelian motives} In this section, we use proposition \ref{motivic} together with the Kuga--Satake correspondence to extend the Grothendieck period conjecture from abelian varieties to some varieties whose motive is -- conjecturally -- abelian. 

Recall that a smooth projective variety $X$ over a subfield of $\C$ is said to be \emph{holomorphic symplectic} if its underlying complex variety is simply connected and if $H^0(X, \Omega^2_X)$ is generated by a global everywhere non-degenerate two-form.

Examples of holomorphic symplectic varieties include Hilbert schemes of points and their deformations, generalized Kummer surfaces and their deformations, as well as two classes of sporadic examples in dimension $6$ and $10$. We refer to \cite{Beauville83} for details.

\begin{theorem}\label{ouf}
\begin{enumerate}
\item Let $X$ be a smooth projective holomorphic symplectic variety over $\Qb$, and assume that the second Betti number of $X$ is at least $4$. Then $GPC^1(X)$ holds.
\item Let $X$ be a smooth cubic hypersurface in $\mathbb P^5_{\Qb}$. Then $GPC^2(X)$ holds.
\end{enumerate}
\end{theorem}

To control the geometry of the projective varieties considered in Theorem \ref{ouf}, we shall rely on the following two classical results. 

\begin{theorem}\label{abelian}
Let $X$ be a smooth projective holomorphic symplectic variety over $\Qb$, and assume that the second Betti number of $X$ is at least $4$. Then there exists an abelian variety $A$ over $\Qb$ and an absolute Hodge injective morphism
\begin{equation}\label{KSmap}
\mathrm{KS} : H^2(X^\an_\C, \Q(1))\lra H^2(A^\an_\C, \Q(1)).
\end{equation}
\end{theorem}

When $X$ is a $K3$ surface, this is in substance the main assertion concerning the Kuga--Satake correspondence in \cite{Deligne72}, which was written before the introduction of the notion of absolute Hodge classes. In the proof of \cite{Andre96b}, Lemme 7.1.3, Andr\'e shows that, in the case of $K3$ surfaces, the Kuga--Satake correspondence is a motivated cycle. The general result -- actually, the fact that the Kuga-Satake correspondence for general holomorphic symplectic is motivated -- is proved in \cite{Andre96}, Corollary 1.5.3 and Proposition 6.2.1; see also \cite{CharlesSchnell11}, 4.5.

Let us only recall that the Kuga--Satake correspondence, first introduced in \cite{KugaSatake67}, is defined analytically through an algebraic group argument at the level of the moduli spaces of holomorphic symplectic varieties and abelian varieties, which are both open subsets of Shimura varieties. It is not known whether it is induced by an algebraic cycle, although this is expected as an instance of the Hodge conjecture.

The second result is due to Beauville and Donagi in \cite{BeauvilleDonagi85}.

\begin{theorem}\label{BD}
Let $X$ be a smooth cubic hypersurface in $\mathbb P^5_{\Qb}$. Then there exists a smooth projective holomorphic symplectic fourfold $F$ over $\Qb$, and an isomorphism
\begin{equation}\label{incidence}
\phi : H^4(X^\an_\C, \Q(2))\lra H^2(F^\an_\C, \Q(1))
\end{equation}
that is induced by an algebraic correspondence between $X$ and $F$.
\end{theorem}

\begin{proof}
While we refer to \cite{BeauvilleDonagi85} for the details of its proof, we briefly indicate the basic geometric constructions behind this theorem.

Let $F$ be the variety of lines in $X$. Beauville and Donagi prove that $F$ is a smooth projective holomorphic symplectic variety of dimension $4$ with second Betti number equal to $23$. 

By the following construction, codimension $2$ cycles on $X$ are related to divisors on $F$. 

Let $Z$ be the incidence correspondence between $F$ and $X$. Points of $Z$ are pairs $(l, x)$ where $l$ is a line in $X$ and $x$ a point of $l$. The incidence correspondence $Z$ maps to both $F$ and $X$ in a tautological way. Since $F$ has dimension $4$, $Z$ has dimension $5$, and the correspondence induces a map
$$H^4(X^\an_\C, \Q(2))\lra H^2(F^\an_\C, \Q(1)).$$

This map is the required isomorphism.
\end{proof}

From the result of Beauville and Donagi, we get the following.

\begin{corollary}\label{cor:incidence}   
Let $X$ be a smooth cubic hypersurface in $\mathbb P^5_{\Qb}$. Then there exists a smooth projective holomorphic symplectic fourfold $F$ over $\Qb$, and an isomorphism
$$H^2(F^\an_\C, \Q(1))\lra H^4(X^\an_\C, \Q(2))$$
that is induced by an algebraic correspondence between $F$ and $X$.
\end{corollary}

\begin{proof}
Let $F$ be as in Theorem \ref{BD}. Let 
$$\psi : H^6(F^\an_\C, \Q(3))\lra H^4(X^\an_\C, \Q(2))$$
be the Poincar\'e dual to $\phi$. It is induced by an algebraic correspondence since $\phi$ is. Let $[h]$ be the cohomology class of a hyperplane section of $F$. By the hard Lefschetz theorem, the map 
$$H^2(F^\an_\C, \Q(1))\lra H^6(F^\an_\C, \Q(3)), \alpha\longmapsto \alpha\cup [h]^2$$
is an isomorphism. It is of course induced by an algebraic correspondence. As a consequence, the map 
$$H^2(F^\an_\C, \Q(1))\lra H^4(X^\an_\C, \Q(2)), \alpha\longmapsto \psi(\alpha\cup[h]^2)$$
is an isomorphism that is induced by an algebraic correspondence.
\end{proof}

\begin{proof}[Proof of Theorem \ref{ouf}]
Given the existence of the Kuga--Satake morphism (\ref{KSmap}) and of the Beauville-Donagi isomorphism (\ref{incidence}), the proposition follows from standard arguments.

\begin{enumerate}
\item We know that $GPC^1(A)$ holds and that $X$ satisfies the Hodge conjecture in codimension $1$ by the Lefschetz (1,1) theorem. Proposition \ref{motivic}, (2) applied to the Kuga--Satake morphism shows that $GPC^1(X)$ holds.
\item Let $F$ be as in Theorem \ref{BD}. Since we just proved that $GPC^1(F)$ holds, Corollary \ref{cor:incidence} and Proposition \ref{motivic}, (3) allows us to conclude.
\end{enumerate}
\end{proof}

\end{document}